\documentclass{amsart}
\usepackage{bbding}
\usepackage{mathrsfs}
\usepackage{mathrsfs}
\usepackage{amsfonts}
\usepackage{cases}
\usepackage{latexsym}
\usepackage{amsmath}
\usepackage[all]{xy}
\usepackage{stmaryrd}
\usepackage{amsmath,amssymb,amscd,bbm,amsthm,mathrsfs,dsfont}
\usepackage{color}
\usepackage{fancyhdr}
\usepackage{amsxtra,ifthen}
\usepackage{verbatim}

\usepackage{fancyhdr}
\usepackage{amsxtra,ifthen}
\usepackage{verbatim}

\usepackage[numbers,sort&compress]{natbib}
\usepackage[pagebackref]{hyperref}
\usepackage{hypernat}

\theoremstyle{plain}
\newtheorem{theorem}{Theorem}[section]

\newtheorem{lemma}[theorem]{Lemma}

\theoremstyle{definition}

\begin{document}

\title[Nonlinear maps preserving Jordan $\eta$-$\ast$-$n$-products]
{Nonlinear maps preserving Jordan $\eta$-$\ast$-$n$-products}

\author{Wenhui Lin}

\address{Lin: College of Science, China Agricultural University, 100083, Beijing, P. R. China}

\email{whlin@cau.edu.cn}

\begin{abstract}
Let $\eta\neq -1$ be a non-zero complex number, and let $\phi$ be a not necessarily linear 
bijection between two von Neumann algebras, one of which has no central abelian projections preserving the Jordan $\eta$-$\ast$-$n$-product. It is showed that $\phi$ is a linear $\ast$-isomorphism if $\eta$ is not real and $\phi$ is the sum of a 
linear $\ast$-isomorphism and a conjugate linear $\ast$-isomorphism if $\eta$ is real.
\end{abstract}




\date{\today}

\maketitle


\section{Introduction}
\label{xxsec1}

Let $\mathcal{A}$ be a $C^{\ast}$-algebra. For a non-zero scalar $\eta \in \Bbb{C}$, the Jordan $\eta$-$\ast$-product of two elelments $A, B\in \mathcal{A}$ is denoted by $A \Diamond_{\eta} B=AB+\eta BA^\ast$. In recent years, an intense research activity has been addressed to study not necessarily linear mappings between von Neumann algebras preserving the $\eta$-$\ast$-product or some of its variants. The origins of the Jordan $\eta$-$\ast$-product go back to  \cite{Semrl}, where P. Semrl introduced and studied the Jordan (-1)-$\ast$-product in relation to quadratic functionals. More recently, Z. Bai and S.P. Du \cite{Bai} established that any bijective map between von Neumann algebras without central abelian projections preserving the  Jordan (-1)-$\ast$-product is a sum of linear and conjugate linear $\ast$-isomorohisms. In \cite{Li}, they proved that a not necessarily linear bijective mapping $\Phi$ between von Neumann algebras preserves the Jordan 1-$\ast$-product if and only if it is a $\ast$-ring isomorphisms. As a corollary, they observe that if the von Neumann algebras are type I factors, then $\Phi$ is a unitary isomorphism or a conjugate unitary isomorphism. In 2014, L. Q. Dai and F. Y. Lu \cite{Dai} generalized the above mentioned result by Bai and Du, by describing all bijective not necessarily linear maps $\Phi$ between two von Neumann algebras, one of which has no central abelian projections, that preserve the Jordan $\eta$-$\ast$-product.The concrete description shows that one of the following statement holds:

\begin{enumerate}
\item[(a)] if $\eta \in \Bbb{R}$, then $\Phi$ is a sum of a linear $\ast$-isomorphism and a conjugate linear $\ast$-isomorphism,

\item[(b)]  if $\eta \notin \Bbb{R}$, then $\Phi$ is a linear $\ast$-isomorphism.
\end{enumerate}

See \cite{Cui, Huo1, Ji, Lu} for other related results. In \cite{Huo2}, they consider the Jordan triple $\eta$-$\ast$-product of three element $A,B$ and $C$ in a $C^{\ast}$-algebra $\mathcal{A}$ defined by  $A \Diamond_{\eta} B \Diamond_{\eta} C=(A \Diamond_{\eta} B) \Diamond_{\eta} C$. A not necessarily linear map $\Phi$ between $C^{\ast}$-algebra $\mathcal{A}$ and $\mathcal{B}$ preserves Jordan triple $\eta$-$\ast$-product if
$$
\Phi(A \Diamond_{\eta} B \Diamond_{\eta} C)=\Phi(A) \Diamond_{\eta} \Phi(B) \Diamond_{\eta} \Phi(C)
$$
for every $A, B, C \in \mathcal{A}$. Clearly a map between $C^{\ast}$-algebra preserving the Jordan $\eta$-$\ast$-product also preserves the Jordan triple $\eta$-$\ast$-product. The main result of \cite{Huo2} proves the following: let $\mathcal{A}$ and $\mathcal{B}$ be two von Neumann algebras, one of which has no central abelian projections, let $\eta \neq -1$ be a non-zero complex number, and let $\phi:\mathcal{A} \rightarrow \mathcal{B}$ be a not necessarily linear bijection with $\phi(I)=I$. Then $\phi$ preserves the Jordan triple $\eta$-$\ast$-product if and only if one of the following statement holds:

\begin{enumerate}
\item[(a)] $\eta \in \Bbb{R}$ and there exists a central projection $p\in \mathcal{A}$ such that $\phi(p)$ is a central projection in $\mathcal{B}$,$\phi\mid_{\mathcal{A}p}:\mathcal{A}p\rightarrow \mathcal{B}\phi(p)$ is a linear $\ast$-isomorphism and $\phi\mid_{\mathcal{A}(I-p)}:\mathcal{A}(I-p)\rightarrow \mathcal{B}(I-\phi(p))$ a conjugate linear $\ast$-isomorphism,

\item[(b)]  $\eta \notin \Bbb{R}$ and $\phi$ is a linear $\ast$-isomorphism.
\end{enumerate}

But Huo et al. \cite{Huo2} prove these conclusions heavily depend on the assumption $\phi(I)=I$. In this paper, we not only generalize the corresponding conclusions to tye-$n$, but also abolish this condition.

Given the consideration of Jordan $\eta$-$\ast$-product and Jordan triple $\eta$-$\ast$-product, we can 
further develop them in one natural way. Suppose that $n\geq 2$ is a fixed positive
integer. Let us see a sequence of polynomials with Jordan $\eta$-$\ast$(where we should be aware that $\Diamond_{\eta}$ is not necessarily associative)
$$
\begin{aligned}
p_2(x_1,x_2)&=x_1\Diamond_{\eta} x_2=x_1x_2+x_2x_1^\ast,\\
p_3(x_1,x_2,x_3)&=p_2(x_1,x_2) \Diamond_{\eta} x_3=(x_1\Diamond_{\eta} x_2) \Diamond_{\eta} x_3\\
&=:x_1\Diamond_{\eta} x_2 \Diamond_{\eta} x_3,\\
p_4(x_1,x_2,x_3,x_4)&=p_3(x_1,x_2,x_3) \Diamond_{\eta} x_4=((x_1\Diamond_{\eta} x_2) \Diamond_{\eta} x_3 \Diamond_{\eta} x_4\\
&=:x_1\Diamond_{\eta} x_2 \Diamond_{\eta} x_3 \Diamond_{\eta} x_4,\\
\cdots &\cdots,\\
p_n(x_1,x_2,\cdots,x_n)&=p_{n-1}(x_1,x_2,\cdots,x_{n-1}) \Diamond_{\eta} x_n\\
&=\underbrace{(\cdots ((}_{n-2}x_1 \Diamond_{\eta} x_2)\Diamond x_3)\Diamond_{\eta} \cdots \Diamond_{\eta} x_{n-1}) \Diamond_{\eta} x_n\\
&=:x_1\Diamond_{\eta} x_2 \Diamond_{\eta} \cdots\Diamond_{\eta} x_n.
\end{aligned}
$$
Accordingly, a \textit{nonlinear map} $\phi:\mathcal{A} \longrightarrow \mathcal{B}$ between $C^{\ast}$-algebra $\mathcal{A}$ and $\mathcal{B}$ preserves Jordan $\eta$-$\ast$-$n$-products if
$$
\phi(x_1\Diamond_{\eta} x_2 \Diamond_{\eta} \cdots\Diamond_{\eta} x_n)=\phi(x_1)\Diamond_{\eta} \phi(x_2) \Diamond_{\eta} \cdots\Diamond_{\eta}\phi( x_n)
$$
for all $x_1,x_2,\cdots,x_n\in \mathcal{A}$. 

In the following of this paper, we usually choose the notation 
$$
\phi(p_n(x_1,x_2,\cdots,x_n))=p_n(\phi(x_1),\phi(x_2),\cdots,\phi(x_n))
$$
instead of the above identity. This notion makes the best use of the definition of Jordan $\eta$-$\ast$-$n$-products. By the definition, it is clear that every Jordan $\eta$-$\ast$-product is a 
$\eta$-$\ast$-$2$-product and every Jordan triple $\eta$-$\ast$-product is a Jordan $\eta$-$\ast$-$3$-product.

Motivated by the afore-mentioned works, we will concentrate on giving a description of nonlinear Jordan $\eta$-$\ast$-$n$-products on von Neumann algebras. The framework of this paper is as follows. We recall and collect some 
indispensable facts with respect to Jordan $\eta$-$\ast$-$n$-products on von Neumann algebras in the second section \ref{xxsec2}. The third 
Section \ref{xxsec3} is to provide a detailed proof the additivity of Jordan $\eta$-$\ast$-$n$-products on von Neumann algebras \ref{6}.  The forth Section \ref{xxsec4} is to prove our main result \ref{Th2}.Let $\eta\neq -1$ be a non-zero complex number, and let $\phi$ be a not necessarily linear 
bijection between two von Neumann algebras, one of which has no central abelian projections preserving the Jordan $\eta$-$\ast$-$n$-product. It is showed that $\phi$ is a linear $\ast$-isomorphism if $\eta$ is not real and $\phi$ is the sum of a 
linear $\ast$-isomorphism and a conjugate linear $\ast$-isomorphism if $\eta$ is real. The last Section \ref{xxsec5} is devoted to certain potential topics in this vein for the future.

\section{Notations and Preliminaries}
\label{xxsec2}

Before beginning detailed demonstration and stating our main result, we need to give some notations and preliminaries. Throughout the paper, all algebras and spaces are defined over the field $\Bbb{C}$ of complex numbers. A von Neumann algebra $\mathcal{A}$ is weakly closed and self-adjoint algebra of operators on a Hilbert space $\mathcal{H}$ containin the identity operator $I$. The set $\mathcal{Z}_{\mathcal{A}}=\{S\in \mathcal{A}\mid ST=TS, \forall T\in \mathcal{A} \}$ is called the center of $\mathcal{A}$. A projection $P$ is called a central abelian projection if $P \in \mathcal{Z}_{\mathcal{A}}$ and $P\mathcal{A}P$ is abelian. For $A \in \mathcal{A}$, the central carrier of $A$, denoted by $\overline{A}$, is the smallest central projection $P$ with $PA=A$. It is not difficult to see that $\overline{A}$ is the projection onto the closed subspace spanned by $\{BAx:B \in \mathcal{A}, x \in  \mathcal{H}\}$. Let $Q$ be a projection in $\mathcal{A}$. The core of $Q$, denoted by $\underline{Q}$, is the biggest central projection $P$ with $P \leqslant Q$. If $\underline{Q}=0$, we then call $Q$ a core-free projection. It is easy to verify that $\underline{Q}=0$ if and only if $\overline{I-Q}=I$, where $I$ is the identity operator. A self-adjoint element $A$ of $\mathcal{A}$ is called positive if its spectrum $\sigma(A)$ consists of non-negative real numbers. Moreover, an element $A$ of $\mathcal{A}$ is called positive if and only if there exists $B$ in $\mathcal{A}$ with $A=B^{\ast}B$. Especially, if $B$ is a self-adjoint operator, then $A=B^2$.

\begin{lemma}\label{1}{\rm \cite[Lemma 14]{Miers}}
Let $\mathcal{A}$ be a von Neumann algebra without central abelian projections. Then there exists a projection $P$ with $\underline{P}=0$ and $\overline{P}=I$.
\end{lemma}

\begin{lemma}\label{2}{\rm \cite[Lemma 1.2]{Dai}}
Let $\mathcal{A}$ be a von Neumann algebra without central abelian projections. Then there exists a projection $P$ with $\underline{P}=0$ and $\overline{P}=I$.
\begin{enumerate}
\item[(1)] If $ABP=0$ for all $B \in \mathcal{A}$, then $A=0$;

\item[(2)] If $\eta$ is a non-zero scalar and $(PT(I-P))\Diamond_{\eta} A=0$ for all $T \in \mathcal{A}$, then $A(I-P)=0$.
\end{enumerate}
\end{lemma}

\begin{lemma}\label{3}
Suppose that $A_1, A_2, \cdots, A_n$ and $T$ are in $\mathcal{A}$ with $\phi(T)=\sum_{i=1}^n \phi(A_{i})$. Then for $S_1, S_2, \cdots, S_n \in \mathcal{A}$, we have 
$$
\begin{aligned}
&\overset{\ \ k}{\phi(S_1\Diamond_{\eta}\cdots \Diamond_{\eta}S_{k-1} \Diamond_{\eta} T \Diamond_{\eta}S_{k+1}\Diamond_{\eta}\cdots\Diamond_{\eta} S_n)}\\
=&\overset{k \ \ }{\phi(S_1)\Diamond_{\eta}\cdots \Diamond_{\eta}\phi(S_{k-1}) \Diamond_{\eta} \phi(T) \Diamond_{\eta}\phi(S_{k+1})\Diamond_{\eta}\cdots\Diamond_{\eta} \phi(S_n))}\\
=&\sum_{i=1}^n {\overset{\ \ k}{\phi(S_1\Diamond_{\eta}\cdots \Diamond_{\eta}S_{k-1} \Diamond_{\eta} A_i \Diamond_{\eta}S_{k+1}\Diamond_{\eta}\cdots\Diamond_{\eta} S_n)}},
\end{aligned}
$$
where $k=1,2,\cdots,n$.
 \end{lemma}

\begin{lemma}\label{4}
Let $\mathcal{A}$ be a von Neumann algebra without central abelian projections. 
For any $A\in \mathcal{A}$ and for any positive integer $n\geq 2$, we have
$$
p_n\left (I,\frac{I}{1+\eta}, \cdots, \frac{I}{1+\eta},\frac{A}{1+\eta} \right)= A. 
$$
and
$$
\phi(A)=p_n\left (\phi(I),\phi\left(\frac{I}{1+\eta}\right), \cdots,\phi\left( \frac{I}{1+\eta} \right),\phi\left(\frac{A}{1+\eta} \right) \right). 
$$
 \end{lemma}

\begin{proof}
A recursive calculation gives that
 $$
 \begin{aligned}
&p_n\left(I,\frac{I}{1+\eta}, \cdots, \frac{I}{1+\eta},\frac{A}{1+\eta} \right )\\
=&p_{n-1}\left(I,\frac{I}{1+\eta}, \cdots, \frac{I}{1+\eta},\frac{A}{1+\eta} \right )\\
=&p_{n-2}\left(I,\frac{I}{1+\eta}, \cdots, \frac{I}{1+\eta},\frac{A}{1+\eta} \right )\\
=&\cdots\\
=&A.
\end{aligned}
\eqno{(1)}
$$
By the definition of Jordan $\eta$-$\ast$-$n$-products, we naturally get
$$
\phi(A)=p_n\left (\phi(I),\phi\left(\frac{I}{1+\eta}\right), \cdots,\phi\left( \frac{I}{1+\eta} \right),\phi\left(\frac{A}{1+\eta} \right) \right). \eqno{(2)} 
$$
\end{proof}

\begin{lemma}
Set $\eta=1$. Let $A\in\mathcal{A}$ such that $A^\ast=-A$ and $Z\in \mathcal{Z}(\mathcal{A})$. Then we have
$$
p_n\left(  x_1,x_2 \cdots,x_{n-2}, A,Z\right)=0
$$
for every $x_1,x_2 \cdots,x_{n-2}\in\mathcal{A}$.
\end{lemma}

\begin{proof}
For each $x_1,x_2 \cdots,x_{n-2}, A\in\mathcal{A}$, $A^\ast=-A$ and $Z\in \mathcal{Z}(\mathcal{A})$, we have
$$
p_n\left(  x_1,x_2 \cdots,x_{n-2}, A,Z\right)=p_{n-2}\left(  x_1,x_2 \cdots,x_{n-2}\right)\Diamond_{\eta}A\Diamond_{\eta}Z.
$$
Now we note $p_{n-2}\left(  x_1,x_2 \cdots,x_{n-2}\right)=M$, then we get
$$
\begin{aligned}
p_n\left(  x_1,x_2 \cdots,x_{n-2}, A,Z\right)&=M\Diamond A\Diamond Z\\
&=(MA+AM^{\ast})\Diamond Z\\
&=(MA+AM^{\ast})Z+Z(A^{\ast}M^{\ast}+MA^{\ast})\\
&=M(A+A^{\ast})Z+(A+A^{\ast})M^{\ast}Z\\
&=0.
\end{aligned}
\eqno{(3)}
$$
\end{proof}

\begin{lemma}\label{5}
Set $\eta=1$. If we note $M_n=p_n\left( \phi(I), \cdots,\phi(I)\right)$, then we have $M_{n}+M_{n}^{\ast}=2^{n}I$ for every $n \geq 1$.
\end{lemma}

\begin{proof}
Since $\phi$ is injective, there exsit $B\in\mathcal{A}$ satisfying $\phi(B)=I$. Then we have
$$
\begin{aligned}
2^{n}I=2^{n}\phi(B)=&\phi \left( p_{n+1}\left(  I, \cdots,I, B\right)\right)\\
=& p_{n+1}\left( \phi(I),\cdots, \phi(I),\phi(B) \right)\\
=& p_{n+1}\left( \phi(I),\cdots, \phi(I),I \right)\\
=& p_{n}\left( \phi(I),\cdots, \phi(I)\right)\Diamond I\\
=&M_{n}+M_{n}^{\ast}.
\end{aligned}
\eqno{(4)}
$$
If $n=1$, we have 
$$
\phi(I)+\phi(I)^{\ast}=2I. \eqno{(5)}
$$
\end{proof}

We now choose a projection $P_1\in \mathcal{A}$ and let $P_2=I-P_1$. Let us write $\mathcal{A}_{jk}=P_j\mathcal{A}P_k$ for all $j,k=1,2$. Then we have the Peirce decomposition of $\mathcal{A}$ as
$\mathcal{A}=\mathcal{A}_{11} + \mathcal{A}_{12} + \mathcal{A}_{21} +\mathcal{A}_{22}$. Thus an arbitrary operator $A\in \mathcal{A}$ can be written as $A=A_{11}+ A_{12} + A_{21} + A_{22}$, where $A_{jk}\in \mathcal{A}_{jk}$ and $A_{jk}^\ast\in \mathcal{A}_{kj}$.

\section{Additivity}
\label{xxsec3}

\begin{theorem} \label{6}
Let $\mathcal{A}$ be a von Neumann algebra without central abelian projections and  $\mathcal{B}$ be a $\ast$-algebra. Let $\eta$ be a non-zero scalar with $\eta \neq -1$. Suppose that $\phi$ is a bijiective map from $\mathcal{A}$ to $\mathcal{B}$ with 
$$
\phi(p_n(x_1,x_2,\cdots,x_n))=p_n(\phi(x_1),\phi(x_2),\cdots,\phi(x_n))
$$
for all $x_1,x_2,\cdots,x_n \in \mathcal{A}$, then $\phi$ is additive.
\end{theorem}

\begin{proof}
The proof of this theorem will be laid out nicely in several claims.
\vspace{2mm}

\noindent {\bf Claim 1.} $\phi(0)=0$.

Since $\phi$ is a surjective, we can find $A\in \mathcal{A}$ with $\phi(A)=0$, which implies that
$$
\phi(0)=\phi(p_n(0,0,\cdots,0,A))=p_n(\phi(0),\cdots,\phi(0),\cdots,\phi(A))=0.
$$

In order to continue our discussions, we need the Peirce decomposition of $\mathcal{A}$ as
$\mathcal{A}=\mathcal{A}_{11}+\mathcal{A}_{12}+\mathcal{A}_{21}+\mathcal{A}_{22}$. Then for any operator $A\in \mathcal{A}$, we may write
$A=A_{11}+A_{12}+A_{21}+A_{22}$ for any $A_{jk}\in \mathcal{A}_{jk}(j, k=1,2)$.

\noindent{\bf Claim 2.} For any $A_{11}\in \mathcal{A}_{11}, D_{22}\in \mathcal{A}_{22}$, we have $\phi(A_{11}+D_{22})=\phi(A_{11})+\phi(D_{22})$.

Since $\phi$ is a surjective, we can find an element $T=\sum_{i,j=1}^2 T_{ij}$ with $\phi(T)=\phi(A_{11})+\phi(D_{22})$. For any $\lambda \in \Bbb{C}$, 
$p_n\left (I,\frac{I}{1+\eta}, \cdots, \frac{I}{1+\eta},\frac{\lambda P_1}{1+\eta}, D_{22} \right)=0$. Applying Lemma \ref{3} and Claim 1 in Section \ref{xxsec3}, we have
$$
\phi\left( p_n\left (I,\frac{I}{1+\eta}, \cdots, \frac{I}{1+\eta},\frac{\lambda P_1}{1+\eta}, T \right)\right)=\phi\left( p_n\left (I,\frac{I}{1+\eta}, \cdots, \frac{I}{1+\eta},\frac{\lambda P_1}{1+\eta}, A_{11} \right) \right).
$$
By the injectivity of $\phi$, we have 
$$
p_n\left (I,\frac{I}{1+\eta}, \cdots, \frac{I}{1+\eta},\frac{\lambda P_1}{1+\eta}, T \right)=p_n\left (I,\frac{I}{1+\eta}, \cdots, \frac{I}{1+\eta},\frac{\lambda P_1}{1+\eta}, A_{11} \right).
$$
By Eq.(1), we get
$$
\lambda P_1\Diamond_{\eta} T=\lambda P_1\Diamond_{\eta} A_{11},
$$
which implies that
$$
(\lambda+\eta \overline{\lambda})T_{11}+\lambda T_{12}+\eta \overline{\lambda}T_{21}=(\lambda+\eta \overline{\lambda})A_{11}.
$$
Suppose that $\lambda \neq 0$ and $\lambda+\eta \overline{\lambda} \neq 0$, we get $T_{11}=A_{11},T_{12}=T_{21}=0$. 

Similarly, we have $T_{22}=D_{22}$.

\noindent{\bf Claim 3.} For any $B_{12}\in \mathcal{A}_{12}, C_{21}\in \mathcal{A}_{21}$, we have $\phi(B_{12}+C_{21})=\phi(B_{12})+\phi(C_{21})$.

Since $\phi$ is a surjective, we can find an element $T=\sum_{i,j=1}^2 T_{ij}$ with $\phi(T)=\phi(B_{12})+\phi(C_{21})$. For any $\lambda \in \Bbb{C}$, by Eq.(1), since
$$
\begin{aligned}
&p_n\left (I,\frac{I}{1+\eta}, \cdots, \frac{I}{1+\eta},\frac{\lambda \eta P_1-\overline{\lambda}P_2}{1+\eta}, B_{12} \right)\\
=&(\lambda \eta P_1-\overline{\lambda}P_2)\Diamond_{\eta} B_{12}\\
=&0.
\end{aligned}
$$
Applying Lemma \ref{3} and Claim 1 in Section \ref{xxsec3} again, we have
$$
\begin{aligned}
&\phi\left( p_n\left (I,\frac{I}{1+\eta}, \cdots, \frac{I}{1+\eta},\frac{\lambda \eta P_1-\overline{\lambda}P_2}{1+\eta}, T \right)\right)\\
=&\phi\left( p_n\left( I,\frac{I}{1+\eta}, \cdots, \frac{I}{1+\eta},\frac{\lambda \eta P_1-\overline{\lambda}P_2}{1+\eta}, C_{21} \right) \right).
\end{aligned}
$$
By the injectivity of $\phi$, we have 
$$
p_n\left (I,\frac{I}{1+\eta}, \cdots, \frac{I}{1+\eta},\frac{\lambda \eta P_1-\overline{\lambda}P_2}{1+\eta}, T \right)=p_n\left (I,\frac{I}{1+\eta}, \cdots, \frac{I}{1+\eta},\frac{\lambda \eta P_1-\overline{\lambda}P_2}{1+\eta}, C_{21} \right),
$$
which is equivalent to
$$
(\lambda \eta P_1-\overline{\lambda}P_2)\Diamond_{\eta}T=(\lambda \eta P_1-\overline{\lambda}P_2)\Diamond_{\eta}C_{21}
$$
by Eq.(1).
Then we obtain that
$$
(\lambda \eta+\overline{\lambda} |\eta|^2) T_{11}-(\lambda \eta+\overline{\lambda}) T_{22}+(\overline{\lambda} |\eta |^2-\overline{\lambda})T_{21}=(\overline{\lambda}|\eta|^2-\overline{\lambda})C_{21},
$$
for all that $\lambda \in \Bbb{C}$. Thus we get $T_{11}=T_{22}=0$. 

 For any $\lambda \in \Bbb{C}$, since
$$
p_n\left (I,\frac{I}{1+\eta}, \cdots, \frac{I}{1+\eta},\frac{\lambda B_{12}}{1+\eta}, P_{1} \right)=0.
$$
Applying Lemma \ref{3} and Claim 1 in Section \ref{xxsec3} again, we have
$$
\begin{aligned}
&\phi\left( p_n\left (I,\frac{I}{1+\eta}, \cdots, \frac{I}{1+\eta},\frac{\lambda T}{1+\eta},P_{1} \right)\right)\\
=&\phi\left( p_n\left( I,\frac{I}{1+\eta}, \cdots, \frac{I}{1+\eta},\frac{\lambda C_{21}}{1+\eta}, P_{1} \right) \right).
\end{aligned}
$$
By the injectivity of $\phi$, we have 
$$
p_n\left (I,\frac{I}{1+\eta}, \cdots, \frac{I}{1+\eta},\frac{\lambda T}{1+\eta},P_{1} \right)=p_n\left (I,\frac{I}{1+\eta}, \cdots, \frac{I}{1+\eta},\frac{\lambda C_{21}}{1+\eta}, P_{1} \right),
$$
which implies that
$$
\lambda T_{21}+\overline{\lambda} \eta T_{21}^{\ast}=\lambda C_{21}+\overline{\lambda} \eta C_{21}^{\ast}.
$$
Suppose that $\lambda \neq 0$ and $\eta \overline{\lambda} \neq 0$, then we get $T_{21}=C_{21}$. 

Similarly, we have $T_{12}=B_{12}$.

\noindent{\bf Claim 4.}  For $i,j,k \in \{ 1,2 \}$, if $i \neq j, A_{kk}\in \mathcal{A}_{kk}$ and $B_{ij}\in \mathcal{A}_{ij}$, we have $\phi(A_{kk}+B_{ij})=\phi(A_{kk})+\phi(B_{ij})$.

We only prove the case $i=k=1,j=2$. The proof of other cases is similar. Since $\phi$ is a surjective, we can find an element $T=\sum_{i,j=1}^2 T_{ij}$ with $\phi(T)=\phi(A_{11})+\phi(B_{12})$. For any $\lambda \in \Bbb{C}$, since
$$
p_n\left (I,\frac{I}{1+\eta}, \cdots, \frac{I}{1+\eta},\frac{\lambda P_2}{1+\eta}, A_{11} \right)=0.
$$
Applying Lemma \ref{3} and Claim 1 in Section \ref{xxsec3} again, we have
$$
\begin{aligned}
&\phi\left( p_n\left (I,\frac{I}{1+\eta}, \cdots, \frac{I}{1+\eta},\frac{\lambda P_2}{1+\eta}, T \right)\right)\\
=&\phi\left( p_n\left( I,\frac{I}{1+\eta}, \cdots, \frac{I}{1+\eta},\frac{\lambda P_2}{1+\eta}, B_{12} \right) \right).
\end{aligned}
$$
By the injectivity of $\phi$, we have 
$$
p_n\left (I,\frac{I}{1+\eta}, \cdots, \frac{I}{1+\eta},\frac{\lambda P_2}{1+\eta}, T \right)=p_n\left (I,\frac{I}{1+\eta}, \cdots, \frac{I}{1+\eta},\frac{\lambda P_2}{1+\eta}, B_{12} \right),
$$
which implies that
$$
(\lambda +\overline{\lambda}\eta) T_{22}+\lambda T_{21}+\overline{\lambda}\eta T_{12}=\overline{\lambda}\eta B_{12},
$$
for all $\lambda \neq 0$. Thus we get $T_{21}=T_{22}=0$ and $T_{12}=B_{12}$. 

 For any $\lambda \in \Bbb{C}$, since
$$
p_n\left (I,\frac{I}{1+\eta}, \cdots, \frac{I}{1+\eta},\frac{\lambda \eta P_1-\overline{\lambda}P_2}{1+\eta}, B_{12} \right)=0.
$$
Applying Lemma \ref{3} and Claim 1 in Section \ref{xxsec3} again, we have
$$
\begin{aligned}
&\phi\left( p_n\left (I,\frac{I}{1+\eta}, \cdots, \frac{I}{1+\eta},\frac{\lambda \eta P_1-\overline{\lambda}P_2}{1+\eta}, T \right)\right)\\
=&\phi\left( p_n\left( I,\frac{I}{1+\eta}, \cdots, \frac{I}{1+\eta},\frac{\lambda \eta P_1-\overline{\lambda}P_2}{1+\eta}, A_{11} \right) \right).
\end{aligned}
$$
By the injectivity of $\phi$, we have 
$$
p_n\left (I,\frac{I}{1+\eta}, \cdots, \frac{I}{1+\eta},\frac{\lambda \eta P_1-\overline{\lambda}P_2}{1+\eta}, T \right)=p_n\left (I,\frac{I}{1+\eta}, \cdots, \frac{I}{1+\eta},\frac{\lambda \eta P_1-\overline{\lambda}P_2}{1+\eta}, A_{11} \right).
$$

A similar discussionas the above, we get $T_{11}=A_{11}$.

\noindent {\bf Claim 5.}  For any $A_{11}\in \mathcal{A}_{11},B_{12}\in \mathcal{A}_{12}, C_{21}\in \mathcal{A}_{21}$ and $D_{22}\in \mathcal{A}_{22}$, we have 
$$
\phi(A_{11}+B_{12}+C_{21})=\phi(A_{11})+\phi(B_{12})+\phi(C_{21})
$$
and
$$
\phi(B_{12}+C_{21}+D_{22})=\phi(B_{12})+\phi(C_{21})+\phi(D_{22}).
$$

We just prove the first identity, the second identity can be proved by the similar method.

Since $\phi$ is a surjective, we can find an element $T=\sum_{i,j=1}^2 T_{ij}$ with $\phi(T)=\phi(A_{11})+\phi(B_{12})+\phi(C_{21})$. For any $\lambda \in \Bbb{C}$, applying Lemma \ref{3}, we have
$$
\begin{aligned}
&\phi\left( p_n\left (I,\frac{I}{1+\eta}, \cdots, \frac{I}{1+\eta},\frac{\lambda P_2}{1+\eta}, T \right)\right)\\
=&\phi\left( p_n\left( I,\frac{I}{1+\eta}, \cdots, \frac{I}{1+\eta},\frac{\lambda P_2}{1+\eta}, A_{11} \right) \right)\\
&+\phi\left( p_n\left( I,\frac{I}{1+\eta}, \cdots, \frac{I}{1+\eta},\frac{\lambda P_2}{1+\eta}, B_{12} \right) \right)\\
&+\phi\left( p_n\left( I,\frac{I}{1+\eta}, \cdots, \frac{I}{1+\eta},\frac{\lambda P_2}{1+\eta}, C_{21} \right) \right).
\end{aligned}
$$
By the injectivity of $\phi$, we have 
$$
\begin{aligned}
& p_n\left (I,\frac{I}{1+\eta}, \cdots, \frac{I}{1+\eta},\frac{\lambda P_2}{1+\eta}, T \right)\\
=& p_n\left( I,\frac{I}{1+\eta}, \cdots, \frac{I}{1+\eta},\frac{\lambda P_2}{1+\eta}, A_{11} \right) \\
&+ p_n\left( I,\frac{I}{1+\eta}, \cdots, \frac{I}{1+\eta},\frac{\lambda P_2}{1+\eta}, B_{12} \right) \\
&+ p_n\left( I,\frac{I}{1+\eta}, \cdots, \frac{I}{1+\eta},\frac{\lambda P_2}{1+\eta}, C_{21} \right) .
\end{aligned}
$$
which implies that
$$
\lambda P_2 \Diamond_{\eta} T=\lambda P_2 \Diamond_{\eta}  A_{11}+\lambda P_2 \Diamond_{\eta} B_{12}+\lambda P_2 \Diamond_{\eta} C_{21}
$$
by Eq.(1), so we have
$$
(\lambda +\eta \overline{\lambda}) T_{22}+\lambda T_{21}+\eta \overline{\lambda} T_{12}=\eta \overline{\lambda} B_{12}+\lambda C_{21},
$$
which implies that $T_{22}=0,T_{12}=B_{12},T_{21}=C_{21}$. Thus we get $T=T_{11}+B_{12}+C_{21}$.

Similarly, we have
$$
\begin{aligned}
&\phi\left( p_n\left (I,\frac{I}{1+\eta}, \cdots, \frac{I}{1+\eta},\frac{\lambda \eta P_1-\overline{\lambda} P_2}{1+\eta}, T \right)\right)\\
=&\phi\left( p_n\left( I,\frac{I}{1+\eta}, \cdots, \frac{I}{1+\eta},\frac{\lambda \eta P_1-\overline{\lambda} P_2}{1+\eta}, A_{11} \right) \right)\\
&+\phi\left( p_n\left( I,\frac{I}{1+\eta}, \cdots, \frac{I}{1+\eta},\frac{\lambda \eta P_1-\overline{\lambda} P_2}{1+\eta}, B_{12} \right) \right)\\
&+\phi\left( p_n\left( I,\frac{I}{1+\eta}, \cdots, \frac{I}{1+\eta},\frac{\lambda \eta P_1-\overline{\lambda} P_2}{1+\eta}, C_{21} \right) \right).
\end{aligned}
$$
By the injectivity of $\phi$, we have 
$$
\begin{aligned}
& p_n\left (I,\frac{I}{1+\eta}, \cdots, \frac{I}{1+\eta},\frac{\lambda \eta P_1-\overline{\lambda} P_2}{1+\eta}, T \right)\\
=& p_n\left( I,\frac{I}{1+\eta}, \cdots, \frac{I}{1+\eta},\frac{\lambda \eta P_1-\overline{\lambda} P_2}{1+\eta}, A_{11} \right) \\
&+ p_n\left( I,\frac{I}{1+\eta}, \cdots, \frac{I}{1+\eta},\frac{\lambda \eta P_1-\overline{\lambda} P_2}{1+\eta}, B_{12} \right) \\
&+ p_n\left( I,\frac{I}{1+\eta}, \cdots, \frac{I}{1+\eta},\frac{\lambda \eta P_1-\overline{\lambda} P_2}{1+\eta}, C_{21} \right) .
\end{aligned}
$$
which implies that
$$
(\lambda \eta P_1-\overline{\lambda} P_2) \Diamond_{\eta} T=(\lambda \eta P_1-\overline{\lambda} P_2) \Diamond_{\eta}  A_{11}+(\lambda \eta P_1-\overline{\lambda} P_2) \Diamond_{\eta} B_{12}+(\lambda \eta P_1-\overline{\lambda} P_2) \Diamond_{\eta} C_{21},
$$
by Eq.(1), so we have
$$
(\lambda \eta + \overline{\lambda}|\eta|^2) T_{11}- \overline{\lambda} (1-|\eta|^2) T_{12}=(\lambda \eta + \overline{\lambda}|\eta|^2) A_{11}- \overline{\lambda} (1-|\eta|^2) B_{12},
$$
then we get $T_{11}=A_{11}$.

\noindent{\bf Claim 6.}  For any $A_{ij}, B_{ij}\in \mathcal{A}_{ij} (1 \leqslant i\neq j \leqslant 2)$, we have $\phi(A_{ij}+B_{ij})=\phi(A_{ij})+\phi(B_{ij})$.

Since 
$$
\begin{aligned}
&p_n\left( I,\frac{I}{1+\eta}, \cdots, \frac{I}{1+\eta},\frac{P_i+A_{ij}}{1+\eta}, P_j+B_{ij} \right)\\
=&\left( P_i+A_{ij}\right) \Diamond_{\eta} (P_j+B_{ij})\\
=&A_{ij}+B_{ij}+\eta(A_{ij}^{\ast}+B_{ij}A_{ij}^{\ast}).
\end{aligned}
$$
By Claim 4 and Claim 5 in Section \ref{xxsec3}, we obtain that 
$$
\begin{aligned}
&\phi(A_{ij}+B_{ij})+\phi(\eta A_{ij}^{\ast})+\phi(\eta B_{ij}A_{ij}^{\ast})\\
=&\phi \left( p_n\left( I,\frac{I}{1+\eta}, \cdots, \frac{I}{1+\eta},\frac{P_i+A_{ij}}{1+\eta}, P_j+B_{ij} \right)\right)\\
=&p_n\left(\phi \left( I\right),\phi \left( \frac{I}{1+\eta}\right), \cdots, \phi \left( \frac{I}{1+\eta}\right),\phi \left( \frac{P_i+A_{ij}}{1+\eta}\right),\phi \left(  P_j+B_{ij} \right)\right)\\
=&p_n\left(\phi \left( I\right),\phi \left( \frac{I}{1+\eta}\right), \cdots, \phi \left( \frac{I}{1+\eta}\right),\phi \left( \frac{P_i}{1+\eta}\right)+\phi \left( \frac{A_{ij}}{1+\eta}\right),\phi \left(  P_j \right)+\phi \left( B_{ij} \right)\right)\\
\end{aligned}
$$
$$
\begin{aligned}
=&p_n\left(\phi \left( I\right),\phi \left( \frac{I}{1+\eta}\right), \cdots, \phi \left( \frac{I}{1+\eta}\right),\phi \left( \frac{P_i}{1+\eta}\right),\phi \left(  P_j \right)\right)\\
&+p_n\left(\phi \left( I\right),\phi \left( \frac{I}{1+\eta}\right), \cdots, \phi \left( \frac{I}{1+\eta}\right),\phi \left( \frac{P_i}{1+\eta}\right),\phi \left( B_{ij} \right)\right)\\
&+p_n\left(\phi \left( I\right),\phi \left( \frac{I}{1+\eta}\right), \cdots, \phi \left( \frac{I}{1+\eta}\right),\phi \left( \frac{A_{ij}}{1+\eta}\right),\phi \left(  P_j \right)\right)\\
&+p_n\left(\phi \left( I\right),\phi \left( \frac{I}{1+\eta}\right), \cdots, \phi \left( \frac{I}{1+\eta}\right),\phi \left( \frac{A_{ij}}{1+\eta}\right),\phi \left( B_{ij} \right)\right)\\
=&\phi \left( p_n\left( I,\frac{I}{1+\eta}, \cdots, \frac{I}{1+\eta},\frac{P_i}{1+\eta}, P_j \right)\right)\\
&+\phi \left( p_n\left( I,\frac{I}{1+\eta}, \cdots, \frac{I}{1+\eta},\frac{P_i}{1+\eta}, B_{ij} \right)\right)\\
&+\phi \left( p_n\left( I,\frac{I}{1+\eta}, \cdots, \frac{I}{1+\eta},\frac{A_{ij}}{1+\eta}, P_j \right)\right)\\
&+\phi \left( p_n\left( I,\frac{I}{1+\eta}, \cdots, \frac{I}{1+\eta},\frac{A_{ij}}{1+\eta}, B_{ij} \right)\right)\\
=&\phi(A_{ij})+\phi(B_{ij})+\phi(\eta A_{ij}^{\ast})+\phi(\eta B_{ij}A_{ij}^{\ast}).
\end{aligned}
$$
Thus we have $\phi(A_{ij}+B_{ij})=\phi(A_{ij})+\phi(B_{ij})$.

\noindent{\bf Claim 7.}  For every $A_{ii},B_{ii}\in \mathcal{A}_{ii}, 1  \leqslant i  \leqslant 2$, we have $\phi(A_{ii}+B_{ii})=\phi(A_{ii})+\phi(B_{ii})$.

Since $\phi$ is a surjective, we can find an element $T=\sum_{i,j=1}^2 T_{ij}$ with $\phi(T)=\phi(A_{ii})+\phi(B_{ii})$. For any $\lambda \in \Bbb{C} $ and $1  \leqslant i \neq k  \leqslant 2$, we have
$$
\begin{aligned}
&\phi\left( p_n\left (I,\frac{I}{1+\eta}, \cdots, \frac{I}{1+\eta},\frac{\lambda P_k}{1+\eta}, T \right)\right)\\
=&\phi\left( p_n\left( I,\frac{I}{1+\eta}, \cdots, \frac{I}{1+\eta},\frac{\lambda P_k}{1+\eta}, A_{ii} \right) \right)\\
&+\phi\left( p_n\left( I,\frac{I}{1+\eta}, \cdots, \frac{I}{1+\eta},\frac{\lambda P_k}{1+\eta}, B_{ii} \right) \right)\\
=&0.
\end{aligned}
$$
By the injectivity of $\phi$, we have 
$$
p_n\left (I,\frac{I}{1+\eta}, \cdots, \frac{I}{1+\eta},\frac{\lambda P_k}{1+\eta}, T \right)=0,
$$
which implies that
$$
(\lambda +\overline{\lambda}\eta) T_{kk}+\lambda T_{ki}+\overline{\lambda}\eta T_{ik}=0,
$$
for all $\lambda \neq 0$. Thus we get $T_{kk}=T_{ki}=T_{ik}=0$. Now we get $T=T_{ii}$. 

 For every $C_{ik} \in \mathcal{A}_{ik},i \neq k$, it follows from Lemma \ref{3} and Claim 6 in Section \ref{xxsec3} that
$$
\begin{aligned}
\phi((\lambda+\eta \overline{\lambda})T_{ii}C_{ik})=&\phi\left( p_n\left (I,\frac{I}{1+\eta}, \cdots, \frac{I}{1+\eta},\frac{\lambda P_i}{1+\eta}, T, C_{ik} \right)\right)\\
=&\phi\left( p_n\left( I,\frac{I}{1+\eta}, \cdots, \frac{I}{1+\eta},\frac{\lambda P_i}{1+\eta}, A_{ii}, C_{ik} \right) \right)\\
&+\phi\left( p_n\left( I,\frac{I}{1+\eta}, \cdots, \frac{I}{1+\eta},\frac{\lambda P_i}{1+\eta}, B_{ii}, C_{ik}\right) \right)\\
=&\phi((\lambda+\eta \overline{\lambda})(A_{ii}C_{ik}+B_{ii}C_{ik})).
\end{aligned}
$$
Hence we have 
$$
(T_{ii-}A_{ii}-B_{ii})C_{ik}=0
$$
for all $C_{ik} \in \mathcal{A}_{ik}$, that is, $(T_{ii}-A_{ii}-B_{ii})CP_{i}=0$ for all $C \in \mathcal{A}$. By Lemma ?, we get that  $T_{ii}=A_{ii}+B_{ii}$. Consequently,
$$
\phi(A_{ii}+B_{ii})=\phi(A_{ii})+\phi(B_{ii}).
$$

\noindent{\bf Claim 8.}  For any $T_{12}, A_{12}, B_{12}\in \mathcal{A}_{12}$ and $A_{21}, B_{21}\in \mathcal{A}_{21}$, we have 
$$
\phi(T_{12}A_{21}+T_{12}B_{21}+\eta A_{12}T_{12}^{\ast}+\eta B_{12}T_{12}^{\ast})=\phi(T_{12}A_{21})+\phi(T_{12}B_{21})+\phi(\eta A_{12}T_{12}^{\ast})+\phi(\eta B_{12}T_{12}^{\ast}).
$$

By Claim 3 and Claim 6 in Section \ref{xxsec3}, we obtain that 
$$
\begin{aligned}
&\phi(T_{12}A_{21}+T_{12}B_{21}+\eta A_{12}T_{12}^{\ast}+\eta B_{12}T_{12}^{\ast})\\
=&\phi \left( p_n\left( I,\frac{I}{1+\eta}, \cdots, \frac{I}{1+\eta},\frac{T_{12}}{1+\eta}, A_{21}+B_{21}+A_{12}+B_{12} \right)\right)\\
=&p_n\left(\phi \left( I\right),\phi \left( \frac{I}{1+\eta}\right), \cdots, \phi \left( \frac{I}{1+\eta}\right),\phi \left( \frac{T_{12}}{1+\eta}\right),\phi \left(  A_{21}+B_{21}+A_{12}+B_{12} \right)\right)\\
=&p_n\left(\phi \left( I\right),\phi \left( \frac{I}{1+\eta}\right), \cdots, \phi \left( \frac{I}{1+\eta}\right),\phi \left( \frac{T_{12}}{1+\eta}\right),\phi \left( A_{21}\right)+\phi \left( B_{21} \right)+\phi \left(A_{12}\right)+\phi \left(B_{12} \right)\right)\\
=&p_n\left(\phi \left( I\right),\phi \left( \frac{I}{1+\eta}\right), \cdots, \phi \left( \frac{I}{1+\eta}\right),\phi \left( \frac{T_{12}}{1+\eta}\right),\phi \left( A_{21}\right)\right)\\
&+p_n\left(\phi \left( I\right),\phi \left( \frac{I}{1+\eta}\right), \cdots, \phi \left( \frac{I}{1+\eta}\right),\phi \left( \frac{T_{12}}{1+\eta}\right),\phi \left( B_{21} \right)\right)\\
&+p_n\left(\phi \left( I\right),\phi \left( \frac{I}{1+\eta}\right), \cdots, \phi \left( \frac{I}{1+\eta}\right),\phi \left( \frac{T_{12}}{1+\eta}\right),\phi \left(A_{12}\right)\right)\\
&+p_n\left(\phi \left( I\right),\phi \left( \frac{I}{1+\eta}\right), \cdots, \phi \left( \frac{I}{1+\eta}\right),\phi \left( \frac{T_{12}}{1+\eta}\right),\phi \left(B_{12} \right)\right)\\
=&\phi \left( p_n\left( I,\frac{I}{1+\eta}, \cdots, \frac{I}{1+\eta},\frac{T_{12}}{1+\eta}, A_{21} \right)\right)+\phi \left( p_n\left( I,\frac{I}{1+\eta}, \cdots, \frac{I}{1+\eta},\frac{T_{12}}{1+\eta}, B_{21} \right)\right)\\
&+\phi \left( p_n\left( I,\frac{I}{1+\eta}, \cdots, \frac{I}{1+\eta},\frac{T_{12}}{1+\eta}, A_{12} \right)\right)+\phi \left( p_n\left( I,\frac{I}{1+\eta}, \cdots, \frac{I}{1+\eta},\frac{T_{12}}{1+\eta}, B_{12} \right)\right)\\
=&\phi(T_{12}A_{21})+\phi(T_{12}B_{21})+\phi(\eta A_{12}T_{12}^{\ast})+\phi(\eta B_{12}T_{12}^{\ast}).
\end{aligned}
$$

\noindent{\bf Claim 9.}  For any $A, B\in \mathcal{A}$ and $T_{12}\in \mathcal{A}_{12}$, we have 
$$
\begin{aligned}
&\phi \left( p_n\left( I,\frac{I}{1+\eta}, \cdots, \frac{I}{1+\eta},\frac{T_{12}}{1+\eta}, A+B \right)\right)\\
=&\phi \left( p_n\left( I,\frac{I}{1+\eta}, \cdots, \frac{I}{1+\eta},\frac{T_{12}}{1+\eta}, A \right)\right)+\phi \left( p_n\left( I,\frac{I}{1+\eta}, \cdots, \frac{I}{1+\eta},\frac{T_{12}}{1+\eta}, B \right)\right)\\
\end{aligned}
$$

We can write $A$ and $B$ as $A=\sum_{i,j=1}^{2}A_{ij}$ and $B=\sum_{i,j=1}^{2}B_{ij}$.By Claims 5, 6 and 8 in Section \ref{xxsec3}, we obtain that 
$$
\begin{aligned}
&\phi \left( p_n\left( I,\frac{I}{1+\eta}, \cdots, \frac{I}{1+\eta},\frac{T_{12}}{1+\eta}, A+B \right)\right)\\
=&\phi(T_{12}A_{21}+T_{12}B_{21}+\eta A_{12}T_{12}^{\ast}+\eta B_{12}T_{12}^{\ast}+T_{12}A_{22}+T_{12}B_{22}+\eta A_{22}T_{12}^{\ast}+\eta B_{22}T_{12}^{\ast})\\
=&\phi(T_{12}A_{21}+T_{12}B_{21}+\eta A_{12}T_{12}^{\ast}+\eta B_{12}T_{12}^{\ast})+\phi(T_{12}A_{22}+T_{12}B_{22})+\phi(\eta A_{22}T_{12}^{\ast}+\eta B_{22}T_{12}^{\ast}))\\
\end{aligned}
$$
$$
\begin{aligned}
=&\phi(T_{12}A_{21})+\phi(T_{12}B_{21})+\phi(\eta A_{12}T_{12}^{\ast})+\phi(\eta B_{12}T_{12}^{\ast})+\phi(T_{12}A_{22})+\phi(T_{12}B_{22}))\\
&+\phi(\eta A_{22}T_{12}^{\ast})+\phi(\eta B_{22}T_{12}^{\ast}))\\
=&\phi(T_{12}A_{21}+\eta A_{12}T_{12}^{\ast})+\phi(T_{12}A_{22})+\phi(\eta A_{22}T_{12}^{\ast})+\phi(T_{12}B_{21}+\eta B_{12}T_{12}^{\ast})\\
&+\phi(T_{12}B_{22}))+\phi(\eta B_{22}T_{12}^{\ast}))\\
=&\phi(T_{12}A_{21}+\eta A_{12}T_{12}^{\ast})+T_{12}A_{22})+\eta A_{22}T_{12}^{\ast})+\phi(T_{12}B_{21}+\eta B_{12}T_{12}^{\ast})+T_{12}B_{22}))+\eta B_{22}T_{12}^{\ast}))\\
=&\phi \left( p_n\left( I,\frac{I}{1+\eta}, \cdots, \frac{I}{1+\eta},\frac{T_{12}}{1+\eta}, A \right)\right)+\phi \left( p_n\left( I,\frac{I}{1+\eta}, \cdots, \frac{I}{1+\eta},\frac{T_{12}}{1+\eta}, B \right)\right).
\end{aligned}
$$

We now ready to prove Theorem \ref{6}. For $A, B \in \mathcal{A}$, we can find $T \in \mathcal{A}$ such that $\phi(T)=\phi(A)+\phi(B)$. By Lemma \ref{1}, there exists a projection $P$ with $\underline{P}=0$ and $\overline{P}=I$. For any $S \in \mathcal{A}$, by Lemma \ref{3} and Claim 9 in Section \ref{xxsec3}, we have
$$
\begin{aligned}
&\phi \left( p_n\left( I,\frac{I}{1+\eta}, \cdots, \frac{I}{1+\eta},\frac{PS(I-P)}{1+\eta}, T \right)\right)\\
=&\phi \left( p_n\left( I,\frac{I}{1+\eta}, \cdots, \frac{I}{1+\eta},\frac{PS(I-P)}{1+\eta}, A \right)\right)\\
&+\phi \left( p_n\left( I,\frac{I}{1+\eta}, \cdots, \frac{I}{1+\eta},\frac{PS(I-P)}{1+\eta}, B \right)\right)\\
=&\phi \left( p_n\left( I,\frac{I}{1+\eta}, \cdots, \frac{I}{1+\eta},\frac{PS(I-P)}{1+\eta}, A+B \right)\right),
\end{aligned}
$$
which implies that
$$
p_n\left( I,\frac{I}{1+\eta}, \cdots, \frac{I}{1+\eta},\frac{PS(I-P)}{1+\eta}, T \right)=p_n\left( I,\frac{I}{1+\eta}, \cdots, \frac{I}{1+\eta},\frac{PS(I-P)}{1+\eta}, A+B\right).
$$
Thus we have $T(I-P)=(A+B)(I-P)$ by Lemma \ref{2}. 

Similarly, we have $\overline{I-P}=I$ and $\underline{I-P}=0$, and the above argument implies that $TP=(A+B)P$. Consequently, we have $T=A+B$, which completes the proof.

\end{proof}

\section{Linearity}
\label{xxsec4}

\begin{theorem} \label{Th2}
Let $\mathcal{A}$ and $\mathcal{B}$ be two von Neumann algebras, one of which has no central abelian projections, let $\eta \neq -1$ be a non-zero complex number, and let $\phi:\mathcal{A}\longrightarrow \mathcal{B}$ be a not necessarily linear bijection. Then $\phi$ preverves the Jordan $\eta$-$\ast$-$n$-product if and only if one of the following statements holds:

\begin{enumerate}
\item[(a)] $\eta \in \Bbb{R}$ and there exists a central projection $p\in \mathcal{A}$ such that $\phi(p)$ is a central projection in $\mathcal{B}$,$\phi\mid_{\mathcal{A}p}:\mathcal{A}p\rightarrow \mathcal{B}\phi(p)$ is a linear $\ast$-isomorphism and $\phi\mid_{\mathcal{A}(I-p)}:\mathcal{A}(I-p)\rightarrow \mathcal{B}(I-\phi(p))$ a conjugate linear $\ast$-isomorphism,

\item[(b)]  $\eta \notin \Bbb{R}$ and $\phi$ is a linear $\ast$-isomorphism.
\end{enumerate}
\end{theorem}

\begin{proof}
The proof will be organized in the following lemmas. First we not that $\phi$ is additive. In fact, if $\mathcal{A}$ has no central abelian projection, then Theorem \ref{6} shows that $\phi$ is additive. If $\mathcal{B}$ has no central abelian projections, then $\phi^{-1}:\mathcal{B}\longrightarrow \mathcal{A}$ is not necessarily linear bijection which preserves the Jordan $\eta$-$\ast$-$n$-product. Applying Theorem \ref{6} to $\phi^{-1}$, we know that $\phi^{-1}$ is additive and thus $\phi$ is additive. Without loss of generality, we assume that $\mathcal{B}$ has no central abelian projections in the following.

\begin{lemma} \label{7}
There exists a non-zero scalar $\alpha$ satisfying the follwing conditions:
\begin{enumerate}
\item[(a)] $\frac{1}{\alpha}\phi(\alpha P)$ is a projection $\mathcal{B}$ if and only if $P$ is a projection $\mathcal{A}$;

\item[(b)]  For any $A \in \mathcal{A},\phi(A)^{\ast}=\frac{\overline{\alpha}}{\alpha}\phi(A)$ if and only if $A^{\ast}=\frac{\overline{\alpha}}{\alpha}A$;

\item[(c)] $\phi(\alpha I)={\alpha}I$.
\end{enumerate}
Further, $\alpha \notin \Bbb{R}$ when $\eta \notin \Bbb{R}$.
\end{lemma}

\begin{proof}
We prove the result in three cases.

{\bf Case 1.} $\eta =1.$

Choosing $\lambda \in \Bbb{C}\setminus\{0\}$ with $\lambda+\overline{\lambda}=0$.Since $\phi$ is surjective, there exists $B, C \in \mathcal{A}$ such that $\phi(B)=I$ and $\phi(C)=\frac{I}{2}$ (In the following paper, $B, C$ always satisfy their corresponding ability). Then for any $A \in \mathcal{A}$, we have
$$
\begin{aligned}
0=&\phi \left( p_n\left( \lambda I,A, B,C \cdots, C \right)\right)\\
=& p_n\left( \phi(\lambda I),\phi(A),\phi( B),\phi(C) \cdots, \phi(C) \right)\\
=& p_n\left( \phi(\lambda I),\phi(A),I,\frac{I}{2} \cdots, \frac{I}{2} \right)\\
=&\phi(\lambda I)(\phi(A)+\phi(A)^\ast)+(\phi(A)+\phi(A)^\ast)\phi(\lambda I)^\ast.
\end{aligned}
$$
Taking $A=B$ in the above equation, we have $\phi(\lambda I)^\ast =-\phi(\lambda I)$, which implies that $\phi(\lambda I)D=D\phi(\lambda I)$ for all $D=D^\ast \in \mathcal{B}$. Let $D_1=\frac{D+D^\ast}{2}$ and $D_2=\frac{D-D^\ast}{2\text{i}}$. Since $D=D_1+\text{i}D_2$ for all $D \in \mathcal{B}$, we have $\phi(\lambda I)D=D\phi(\lambda I)$. Thus $\phi(\lambda I)\in \mathcal{Z}(\mathcal{B})$. Similarly, we have $\phi^{-1}(\lambda I)\in \mathcal{Z}(\mathcal{A})$.

\noindent{\bf Claim 1.1.} For each $A\in\mathcal{A},\phi(A)^\ast=-\phi(A)$ if and only if $A^\ast=-A.$

Let $A\in\mathcal{A}$ such that $A^\ast=-A.$ Then by Eq.(3), we have
$$
\begin{aligned}
0=&\phi \left( p_n\left(  B,C \cdots, C, A,\phi^{-1}(\lambda I)\right)\right)\\
=& p_n\left( I,\frac{I}{2} \cdots, \frac{I}{2},\phi(A),\lambda I \right)\\
=&2\lambda(\phi(A)+\phi(A)^\ast).
\end{aligned}
$$
Thus we have $\phi(A)^\ast=-\phi(A)$, which proves the sufficiency.

To prove the necessity, we note that $\phi^{-1}$ also preserves the Jordan 1-$\ast$-n-product. Since $\phi$ is injective, there exists $B^{\prime}, C^{\prime} \in \mathcal{B}$ such that $\phi(B^{\prime})=I$ and $\phi^{-1}(C^{\prime})=\frac{I}{2}$ (In the following paper, $B^{\prime}, C^{\prime}$ always satisfy their corresponding ability). If $\phi(A)^\ast=-\phi(A)$, then by Eq.(3), we have
$$
\begin{aligned}
0=&\phi^{-1} \left( p_n\left(  B^{\prime},C^{\prime} \cdots, C^{\prime},\phi(A),\phi(\lambda I)\right)\right)\\
=& p_n\left( I,\frac{I}{2} \cdots, \frac{I}{2}, A,\lambda I \right)\\
=&2\lambda(A+A^\ast),
\end{aligned}
$$
which implies that $A^\ast=-A.$

\noindent{\bf Claim 1.2.} $\phi(\mathcal{Z}(\mathcal{A}))=\mathcal{Z}(\mathcal{B})$.

Let  $Z \in \mathcal{Z}(\mathcal{A})$ be arbitrary. For every  $A^\ast=-A \in \mathcal{A}$, by Eq.(3) we have
 $$
\begin{aligned}
0=&\phi \left( p_n\left(  B,C \cdots, C, A,Z\right)\right)\\
=& p_n\left( I,\frac{I}{2} \cdots, \frac{I}{2},\phi(A),\phi(Z) \right)\\
=&\phi(A)\phi(Z)+\phi(Z)\phi(A)^\ast.
\end{aligned}
$$
That is $\phi(A)\phi(Z)=-\phi(Z)\phi(A)^\ast$ holds true for all $A^\ast=-A \in \mathcal{A}$. Since $\phi$ preserves conjugate self-adjoint elements, it follows that $D\phi(Z)=\phi(Z)D$ holds true for all $D=-D^\ast \in \mathcal{B}$. Since for every $D \in \mathcal{B}$, we have $D=D_1+\text{i}D_2$, where $D_1=\frac{D+D^\ast}{2}$ and $D_2=\frac{D-D^\ast}{2\text{i}}$ are self-conjugate self-adjoint elementd. Hence $D\phi(Z)=\phi(Z)D$ holds true for all $D \in \mathcal{A}$. Then $\phi(Z)\in \mathcal{Z}(\mathcal{B})$, which implies that$\phi(\mathcal{Z}(\mathcal{A}))\subseteq \mathcal{Z}(\mathcal{B})$. Thus $\phi(\mathcal{Z}(\mathcal{A}))=\mathcal{Z}(\mathcal{B})$ by considering $\phi^{-1}$.

In the following we assume $\alpha=1$.

\noindent{\bf Claim 1.3.} Let $P$ be a projection in $\mathcal{A}$ and set $Q_P=\frac{1}{2}(\phi(P)+\phi(P)^\ast)$. Then $Q_P$ is a projection in $\mathcal{B}$ with $\phi(P)=\phi(I)Q_P$.

Let $P$ be a projection in $\mathcal{A}$. Then by Claim 1.2, we have 
$$
\begin{aligned}
2^{n-1}\phi(P)=&\phi \left( p_n\left(  I \cdots,I,P,I\right)\right)\\
=& p_n\left( \phi(I),\cdots, \phi(I),\phi(P),\phi(I) \right)\\
=&M_{n-2}\Diamond \phi(P)\Diamond \phi(I)\\
=&(M_{n-2}\phi(P)+\phi(P)M_{n-2}^{\ast})\Diamond \phi(I)\\
=&\phi(I)(M_{n-2}+M_{n-2}^{\ast})(\phi(P)+\phi(P)^{\ast}).
\end{aligned}
$$
Here, we should notice that $M_{n-2} \in \mathcal{Z}(\mathcal{B})$ if $\phi(\mathcal{Z}(\mathcal{A}))=\mathcal{Z}(\mathcal{B})$ and the additivity of $\phi$.

By Eq.(4), we obtain
$$
2\phi(P)=\phi(I)(\phi(P)+\phi(P)^{\ast})=2\phi(I)Q_P,
$$
that is
$$
\phi(P)=\phi(I)Q_P, \eqno{(6)}
$$
On the other hand, considering $M_{n-2} \in \mathcal{Z}(\mathcal{B})$ and using Eq.(4), we obtain
$$
\begin{aligned}
2^{n-1}\phi(P)=&\phi \left( p_n\left(  I \cdots,I,P,P\right)\right)\\
=& p_n\left( \phi(I),\cdots, \phi(I),\phi(P),\phi(P) \right)\\
=&M_{n-2}\Diamond \phi(P)\Diamond \phi(P)\\
=&(M_{n-2}\phi(P)+\phi(P)M_{n-2}^{\ast})\Diamond \phi(P)\\
=&(M_{n-2}+M_{n-2}^{\ast})\phi(P)(\phi(P)+\phi(P)^{\ast})\\
=&2^{n-1}\phi(P)Q_P\\
\end{aligned}
$$
Substituting Eq.(6) into the above identity, we have
$$
\phi(P)=\phi(I)Q_P^2.
$$
This together with the previous result implies that $Q_P=Q_P^2$. Since $Q_P$ is self-adjoint, $Q_P$ is a projection.

\noindent{\bf Claim 1.4.} Let $P$ be a projection in $\mathcal{A}$. Suppose that $A$ in $\mathcal{A}$ is such that $A=PA(I-P)$. Then $\phi(A)=Q_P\phi(A)+\phi(A)Q_P$.

Noticing $\phi(P)=\phi(I)Q_P$, we have
$$
\begin{aligned}
2^{n-2}\phi(A)=&\phi \left( p_n\left(  I \cdots,I,P,A\right)\right)\\
=& p_n\left( \phi(I),\cdots, \phi(I),\phi(P),\phi(A) \right)\\
=&M_{n-2}\Diamond \phi(P)\Diamond \phi(A)\\
=&(M_{n-2}\phi(P)+\phi(P)M_{n-2}^{\ast})\Diamond \phi(A)\\
=&(M_{n-2}+M_{n-2}^{\ast})(\phi(P)\phi(A)+\phi(A)\phi(P)^{\ast})\\
=&2^{n-2}(\phi(I)Q_P\phi(A)+\phi(A)Q_P\phi(I)^{\ast}).
\end{aligned}
$$
That is
$$
\phi(A)=\phi(I)Q_P\phi(A)+\phi(A)Q_P\phi(I)^{\ast}.
$$
Since $\phi(I)+\phi(I)^\ast=2I$ by Eq.(5) and $\phi(I), \phi(I)^\ast\in \mathcal{Z}(\mathcal{B})$ by Claim 1.2, multiplying both sides of the above equation by $Q_P$ from the left and right respectively, we get that $Q_P\phi(A)Q_P=0$. Multiplying both sides of the above equation by $I-Q_P$ from the left and right respectively, we get that $(I-Q_P)\phi(A)(I-Q_P)=0$. Then we obtain $\phi(A)=Q_P\phi(A)+\phi(A)Q_P$.

\noindent{\bf Claim 1.5.} $\phi(I)=I$.

Since $\mathcal{B}$ has no central abelian projections, by Lemma \ref{1}, we can choose a projection $Q\in \mathcal{B}$ satisfying $\underline{Q}=0$ and $\overline{Q}=I$. Let $B$ be in $\mathcal{B}$ such that $B=QB(I-Q)$. Let $P=\frac{1}{2}\left(\phi^{-1}(Q)+\phi^{-1}(Q)^\ast \right)$. Applying the previous two claims to $\phi^{-1}$, we know that $P$ is a projection and $\phi^{-1}(B)=P\phi^{-1}(B)+\phi^{-1}(B)P^{\ast}$. Moreover, 
$$
\phi(P)=\frac{1}{2}\phi(\phi^{-1}(Q)I+I\phi^{-1}(Q)^{\ast})=\phi(I)Q.
$$
Hence
$$
B=\phi(P\phi^{-1}(B)+\phi^{-1}(B)P^{\ast})=\phi(I)QB+B(\phi(I)Q)^{\ast}=\phi(I)B.
$$
Since such $B$ is arbitrary and $\overline{I-Q}=I$, it follows form Lemma \ref{2} that $(I-\phi(I))Q=0$. Hence sine $I-\phi(I)\in \mathcal{Z}(\mathcal{A})$ and $\overline{Q}=I$, it follows that  $I-\phi(I)=0$, proving the claim.

\noindent{\bf Claim 1.6.} $\phi(A)=\phi(A)^{\ast}$.

By Claim 1.5, we have
$$
\begin{aligned}
2^{n-2}\phi(A+A^{\ast})=&\phi \left( p_n\left(  I \cdots,I,A,I\right)\right)\\
=& p_n\left( \phi(I),\cdots, \phi(I),\phi(A),\phi(I), \right)\\
=& p_n\left( I,\cdots, I,\phi(A),I, \right)\\
=&2^{n-3}I\Diamond \phi(A)\Diamond I\\
=&2^{n-2}(\phi(A)+\phi(A)^{\ast}).
\end{aligned}
$$
We have $\phi(A)^{\ast}=\phi(A^{\ast})=\phi(A)$ if and only if $A^{\ast}=A$. 

{\bf Case 2.} $|\eta|=1$ but $\eta \neq 1$.

Sine $|\eta|=1$, there exists $\alpha\in \Bbb{C}\setminus \{ 0\}$ such that $\alpha+\eta\overline{\alpha}=0$. Take for example $\alpha$ a real multiple of $\text{i}e^{\text{i}\frac{\theta}{2}}$, where $\eta=e^{\text{i}\theta}$. So we can choose such $\alpha$ such that its real part is an entire number privided $\eta \neq 1$.

Note that $\frac{\overline{\alpha}}{\alpha}=-\overline{\eta}$.

\noindent{\bf Claim 2.1.} For each $A\in\mathcal{A},\phi(A)^\ast=-\overline{\eta}\phi(A)$ if and only if $A^\ast=-\overline{\eta}A.$

For any $A\in\mathcal{A}$, we have
$$
\begin{aligned}
0=&\phi \left( p_n\left( \alpha I,A, B,\cdots, B \right)\right)\\
=& p_n\left( \phi(\alpha I),\phi(A),\phi( B),\cdots, \phi(B) \right)\\
=& p_n\left( \phi(\alpha I),\phi(A),I,\cdots, I \right)\\
=& p_{n-1}\left( \phi(\alpha I)\phi(A)+\eta\phi(A)\phi(\alpha I)^{\ast},I,\cdots, I \right)\\
=& p_{n-2}\left( \phi(\alpha I)(\phi(A)+\phi(A)^{\ast})+\eta(\phi(A)+\phi(A)^{\ast})\phi(\alpha I)^{\ast},I,\cdots, I \right)\\
=&(n-2)( \phi(\alpha I)(\phi(A)+\phi(A)^{\ast})+\eta(\phi(A)+\phi(A)^{\ast})\phi(\alpha I)^{\ast}).
\end{aligned}
\eqno{(7)}
$$
Taking $A=B$ in the above equation, we have $\phi(\alpha I)^\ast =-\frac{1}{\eta}\phi(\alpha I) =-\overline{\eta}\phi(\alpha I)$. Then Eq.(7) becomes
$$
\phi(\alpha I)(\phi(A)+\phi(A)^{\ast})-(\phi(A)+\phi(A)^{\ast})\phi(\alpha I)=0,
$$
 which implies that $\phi(\alpha I)D=D\phi(\alpha I)$ for all $D=D^\ast \in \mathcal{B}$.Thus we have $\phi(\alpha I)D=D\phi(\alpha I)$ for all $B\in\mathcal{B}$. So $\phi(\alpha I)\in \mathcal{Z}(\mathcal{B})$. 

Similarly, we have $\phi^{-1}(\alpha I)\in \mathcal{Z}(\mathcal{A})$.

Let $A\in\mathcal{A}$ such that $A^\ast=-\overline{\eta}A.$ Then
$$
\begin{aligned}
0=&\phi \left( p_n\left(  A,\phi^{-1}(\alpha I),B,\cdots,B\right)\right)\\
=& p_n\left( \phi(A),\alpha I,I \cdots, I \right)\\
=& p_{n-1}\left(\alpha\phi(A)+\eta\alpha \phi(A)^{\ast},I,\cdots, I \right)\\
=& p_{n-2}\left((\alpha+\overline{\alpha})\phi(A)+\eta(\alpha+\overline{\alpha})\phi(A)^{\ast},I,\cdots, I \right)\\
=&(n-2)(\alpha+\overline{\alpha})(\phi(A)+\eta\phi(A)^{\ast}).
\end{aligned}
$$
Since $\eta \neq 1$, we have $\alpha+\overline{\alpha}\neq 0$. Thus we have $\phi(A)^\ast=-\overline{\eta}\phi(A)$, which proves the sufficiency.

To prove the necessity, we note that $\phi^{-1}$ also preserves the Jordan $\eta$-$\ast$-$n$-products. Since $\phi$ is injective, there exists $B^{\prime}\in \mathcal{B}$ such that $\phi(I)=B^{\prime}$. If $\phi(A)^\ast=-\overline{\eta}\phi(A)$, then we have
$$
\begin{aligned}
0=&\phi^{-1}\left( p_n\left(\phi(A),\phi(\alpha I),B^{\prime},\cdots,B^{\prime}\right)\right)\\
=& p_n\left( A,\alpha I,I \cdots, I \right)\\
=& p_{n-1}\left(\alpha A+\eta\alpha A^{\ast},I,\cdots, I \right)\\
=& p_{n-2}\left((\alpha+\overline{\alpha})A+\eta(\alpha+\overline{\alpha})A^{\ast},I,\cdots, I \right)\\
=&(n-2)(\alpha+\overline{\alpha})(A+\eta A^{\ast}).
\end{aligned}
$$
which implies that $A^\ast=-\overline{\eta}A.$

\noindent{\bf Claim 2.2.} $\phi(\mathcal{Z}(\mathcal{A}))=\mathcal{Z}(\mathcal{B})$.

Let  $Z \in \mathcal{Z}(\mathcal{A})$ be arbitrary. Suppose that $D$ is a selfadjoint element in $\mathcal{B}$. Then $(\alpha D)^\ast=-\overline{\eta} (\alpha D)$ and hence $\phi^{-1}(\alpha D)^\ast=-\overline{\eta} \phi^{-1}(\alpha D)$ by Claim 2.1. Therefore by Eq.(2),
$$
\begin{aligned}
0=&\phi \left( p_n\left( I,\frac{I}{1+\eta}, \cdots, \frac{I}{1+\eta}, \frac{\phi^{-1}(\alpha D)}{1+\eta},Z\right)\right)\\
=& p_n\left(\phi \left( I\right),\phi \left(\frac{I}{1+\eta}\right), \cdots, \phi \left(\frac{I}{1+\eta}\right), \phi \left(\frac{\phi^{-1}(\alpha D)}{1+\eta}\right),\phi \left(Z \right)\right)\\
=&p_{n-1} \left( \phi \left( I\right),\phi \left(\frac{I}{1+\eta}\right), \cdots, \phi \left(\frac{I}{1+\eta}\right), \phi \left(\frac{\phi^{-1}(\alpha D)}{1+\eta}\right)\right)\Diamond_{\eta}\phi \left(Z \right)\\
=&\alpha D\Diamond_{\eta}\phi(Z)\\
=&\alpha (D\phi(Z)-\phi(Z)D)
\end{aligned}
$$
for all selfadjoint elements $D$. It follows that $\phi(Z)\in\mathcal{Z}(\mathcal{B})$ for all $Z\in \mathcal{Z}(\mathcal{A})$. So $\phi(\mathcal{Z}(\mathcal{A}))\subseteq\mathcal{Z}(\mathcal{B})$. hence $\phi(\mathcal{Z}(\mathcal{A}))=\mathcal{Z}(\mathcal{B})$ by considering $\phi^{-1}$.

\noindent{\bf Claim 2.3.} For each $A\in\mathcal{A},\phi(A)^\ast=-\overline{\eta}^{2}\phi(A)$ if and only if $A^\ast=-\overline{\eta}^{2}A.$

Now we choose $A\in\mathcal{A}$ with $A^\ast=-\overline{\eta}^{2}A.$ Then by Eq.(2) we have
$$
\begin{aligned}
2\phi(A)=&\phi\left(p_n\left(I,\frac{I}{1+\eta}, \cdots, \frac{I}{1+\eta},\frac{A}{\alpha(1+\eta)},\alpha I \right)\right )\\
=&p_n\left(\phi \left(I \right ),\phi \left(\frac{I}{1+\eta} \right ), \cdots, \phi \left(\frac{I}{1+\eta}\right ),\phi \left(\frac{A}{\alpha(1+\eta)}\right ),\phi \left(\alpha I \right)\right) \\
=&p_{n-2}\left(\phi \left(I \right ),\phi \left(\frac{I}{1+\eta} \right ), \cdots, \phi \left(\frac{I}{1+\eta}\right )\right)\Diamond_{\eta}\phi \left(\frac{A}{\alpha(1+\eta)}\right )\Diamond_{\eta}\phi \left(\alpha I \right) \\
=&\phi \left(I \right )\Diamond_{\eta}\phi \left(\frac{A}{\alpha(1+\eta)}\right )\Diamond_{\eta}\phi \left(\alpha I \right)\\
=&\phi \left(\alpha I \right)\phi \left(I \right )\left( \phi \left(\frac{A}{\alpha(1+\eta)}\right )+\phi \left(\frac{A}{\alpha(1+\eta)}\right )^{\ast}\right)\\
&+\eta\phi \left(\alpha I \right)\left( \phi \left(\frac{A}{\alpha(1+\eta)}\right )+\phi \left(\frac{A}{\alpha(1+\eta)}\right )^{\ast}\right)\phi \left(I \right )^{\ast}.
\end{aligned}
$$

Taking the adjoint and noting that $\phi(\alpha I)^\ast =-\overline{\eta}\phi(\alpha I)$, we get
$$
\begin{aligned}
2\phi(A)^{\ast}=&-\overline{\eta}\phi\left(\alpha I\right)\left(\phi \left(\frac{A}{\alpha(1+\eta)}\right )+\phi \left(\frac{A}{\alpha(1+\eta)}\right )^{\ast})\phi(I)^{\ast}\right)\\
&-\overline{\eta}^2\phi\left(\alpha I\right)\phi(I)\left(\phi \left(\frac{A}{\alpha(1+\eta)}\right )+\phi \left(\frac{A}{\alpha(1+\eta)}\right )^{\ast}\right)\\
=&-\overline{\eta}\phi(\alpha I)\frac{1}{\eta}\left(\eta\left(\phi \left(\frac{A}{\alpha(1+\eta)}\right )+\phi \left(\frac{A}{\alpha(1+\eta)}\right )^{\ast}\right)\phi(I)^{\ast}\right)\\
&-\overline{\eta}^2\phi(\alpha I)\left(\phi \left(\frac{A}{\alpha(1+\eta)}\right )+\phi \left(\frac{A}{\alpha(1+\eta)}\right )^{\ast}\right)\\
=&-2\overline{\eta}^2 \phi(A).
\end{aligned}
$$
Thus $\phi(A)^{\ast}=-\overline{\eta}^2 \phi(A).$ By considering $\phi^{-1}$, we establish the claim.

\noindent{\bf Claim 2.4.} $\phi(\alpha I)=\alpha I$.

By a recusion calculation, we have
$$
\begin{aligned}
&p_n\left(A,\frac{I}{1-\overline{\eta}},\cdots,\frac{I}{1-\overline{\eta}}\right)\\
=&p_{n-1}\left(A,\frac{I}{1-\overline{\eta}},\cdots,\frac{I}{1-\overline{\eta}}\right)\\
=&p_{n-2}\left(A,\frac{I}{1-\overline{\eta}},\cdots,\frac{I}{1-\overline{\eta}}\right)\\
=&\cdots\\
=&A.
\end{aligned}
$$
At the same time, by the definition of $\phi$, we also have
$$
\begin{aligned}
\phi(A)=&\phi\left( p_n\left(A,\frac{I}{1-\overline{\eta}},\cdots,\frac{I}{1-\overline{\eta}}\right)\right)\\
=&p_n\left(\phi\left(A\right),\phi\left(\frac{I}{1-\overline{\eta}}\right),\cdots,\phi\left(\frac{I}{1-\overline{\eta}}\right)\right).
\end{aligned}
$$

We assume $A=\phi^{-1}(\text{i}\alpha^2 I)$. Since $(\text{i}\alpha^2 I)^{\ast}=-\overline{\eta}^{2}(\text{i}\alpha^2 I)$, we obtain from Claim 2.2 that $A^{\ast}=-\overline{\eta}^{2}A$. Therefore, noting that $(1-\overline{\eta})\alpha=\alpha+\overline{\alpha}$ is rational, we have
$$
\begin{aligned}
(1-\overline{\eta})\alpha(\text{i}\alpha^2 I)&=\phi((1-\overline{\eta})\alpha A)\\
&=\phi\left( p_n\left(A,\frac{I}{1-\overline{\eta}},\cdots,\frac{I}{1-\overline{\eta}},\alpha I\right)\right)\\
&=p_n\left(\phi\left(A\right),\phi\left(\frac{I}{1-\overline{\eta}}\right),\cdots,\phi\left(\frac{I}{1-\overline{\eta}}\right),\phi\left(\alpha I\right)\right)\\
&=\phi\left(A\right)\Diamond_{\eta}\phi\left(\alpha I\right)\\
&=(1-\overline{\eta})\phi\left(A\right)\phi\left(\alpha I\right)\\
&=(1-\overline{\eta})(\text{i}\alpha^2 I)\phi\left(\alpha I\right).
\end{aligned}
$$
Since $\eta\neq 1$, it follows that $\phi(\alpha I)=\alpha I$.

\noindent{\bf Claim 2.5.} Let $P$ be in $\mathcal{A}$, then $\frac{1}{\alpha}\phi(\alpha P)$ is a projection in $\mathcal{B}$ if and only if $P$ is a projection in $\mathcal{A}$.

To prove the sufficiency. We suppose that $P$ is a projection in $\mathcal{A}$. Since $(\alpha P)^{\ast}=-\overline{\eta}(\alpha P)$, it follows Claim 2.1 that $\phi(\alpha P)^{\ast}=-\overline{\eta}\phi((\alpha P))$. Hence $\left(\frac{1}{\alpha}\phi(\alpha P)\right)^{\ast}=-\frac{\overline{\eta}}{\overline{\alpha}}\phi(\alpha P)=\frac{1}{\alpha}\phi(\alpha P)$, i.e., $\frac{1}{\alpha}\phi(\alpha P)$ is selfadjoint.

It remains to show that $\frac{1}{\alpha}\phi(\alpha P)$ is idempotent. Since $(\alpha P^{\perp})^{\ast}=-\overline{\eta}(\alpha P^{\perp})$, where $P^{\perp}=I-P$, we have $\phi(\alpha P^{\perp})^{\ast}=-\overline{\eta}\phi((\alpha P^{\perp}))$. Hence by Eq.(2), we have
$$
\begin{aligned}
0=&\phi \left( p_n\left( I,\frac{I}{1+\eta}, \cdots, \frac{I}{1+\eta}, \frac{\alpha P^{\perp}}{1+\eta},P\right)\right)\\
=& p_n\left( \phi \left(I\right),\phi \left(\frac{I}{1+\eta}\right), \cdots, \phi \left(\frac{I}{1+\eta}\right), \phi \left(\frac{\alpha P^{\perp}}{1+\eta}\right),\phi \left( P\right)\right)\\
=& p_{n-1}\left( \phi \left(I\right),\phi \left(\frac{I}{1+\eta}\right), \cdots, \phi \left(\frac{I}{1+\eta}\right), \phi \left(\frac{\alpha P^{\perp}}{1+\eta}\right)\right)\Diamond_{\eta}\phi \left( P\right)\\
=&\phi(\alpha P^{\perp})\Diamond_{\eta}\phi(P)\\
=&\phi(\alpha P^{\perp})\phi(P)-\phi(P)\phi(\alpha P^{\perp}).
\end{aligned}
$$
So $\phi(\alpha P^{\perp})\phi(P)=\phi(P)\phi(\alpha P^{\perp})$. Taking the adjoint, we get $\phi(P)^{\ast}\phi(\alpha P^{\perp})=\phi(\alpha P^{\perp})\phi(P)^{\ast}$. Hence by Eq.(2),
$$
\begin{aligned}
0=&\phi \left( p_n\left( I,\frac{I}{1+\eta}, \cdots, \frac{I}{1+\eta}, \frac{P}{1+\eta},\alpha P^{\perp}\right)\right)\\
=& p_n\left( \phi \left(I\right),\phi \left(\frac{I}{1+\eta}\right), \cdots, \phi \left(\frac{I}{1+\eta}\right), \phi \left(\frac{P}{1+\eta}\right),\phi \left( \alpha P^{\perp}\right)\right)\\
=& p_{n-1}\left( \phi \left(I\right),\phi \left(\frac{I}{1+\eta}\right), \cdots, \phi \left(\frac{I}{1+\eta}\right), \phi \left(\frac{P}{1+\eta}\right)\right)\Diamond_{\eta}\phi \left( \alpha P^{\perp}\right)\\
=&\phi(P)\Diamond_{\eta}\phi(\alpha P^{\perp})\\
=&\phi(P)\phi(\alpha P^{\perp})+\eta\phi(\alpha P^{\perp})\phi(P)^{\ast}\\
=&\phi(\alpha P^{\perp})(\phi(P)+\phi(P)^{\ast}).
\end{aligned}
\eqno{(8)}
$$
Since $\eta\neq -1$, we can set $\beta=\frac{\alpha}{1+\eta}$. Then noting $\phi(\beta I)\in \mathcal{Z}(\mathcal{B})$, by Eq.(2) we have
$$
\begin{aligned}
\phi(\alpha P)=&\phi \left( p_n\left( I,\frac{I}{1+\eta}, \cdots, \frac{I}{1+\eta}, \frac{P}{1+\eta},\beta I\right)\right)\\
=& p_n\left( \phi \left(I\right),\phi \left(\frac{I}{1+\eta}\right), \cdots, \phi \left(\frac{I}{1+\eta}\right), \phi \left(\frac{P}{1+\eta}\right),\phi \left( \beta I\right)\right)\\
=& p_{n-1}\left( \phi \left(I\right),\phi \left(\frac{I}{1+\eta}\right), \cdots, \phi \left(\frac{I}{1+\eta}\right), \phi \left(\frac{P}{1+\eta}\right)\right)\Diamond_{\eta}\phi \left( \beta I\right)\\
=&\phi(P)\Diamond_{\eta}\phi(\beta I)\\
=&\phi(\beta I)(\phi(P)+\eta\phi(P)^{\ast}).
\end{aligned}
$$
This toegether with Eq.(8) implies that $\phi(\alpha P^{\perp})\phi(\alpha P)=0$. Hence
$$
\phi(\alpha P)^2=(\phi(\alpha P)+\phi(\alpha P^{\perp}))\phi(\alpha P)=\phi(\alpha I)\phi(\alpha P)=\alpha \phi(\alpha P).
$$
So $\frac{1}{\alpha}\phi(\alpha P)$ is idempotent.

So far we have established the sufficienty. Note that the preceding proof does not use the condition that $\mathcal{B}$ has no central abelian projection. Therefore the previous result can apply to $\phi^{-1}$. Now, if $\frac{1}{\alpha}\phi(\alpha P)$ is a projection, then $P=\frac{1}{\alpha}\phi^{-1}(\alpha(\frac{1}{\alpha}\phi(\alpha P)))$ is a projection, proving the necessity.

{\bf Case 3.} $|\eta|\neq 1$.

Take $\alpha=\frac{1-\eta}{1-|\eta|^2}$, then $\alpha+\eta\overline{\alpha}=1$.

\noindent{\bf Claim 3.1.} $\phi(\alpha I)=\alpha I$.

By Eq.(2), we have
$$
\begin{aligned}
I=\phi(B)=&\phi \left( p_n\left( I,\frac{I}{1+\eta}, \cdots, \frac{I}{1+\eta}, \frac{\alpha I}{1+\eta},B\right)\right)\\
=& p_n\left( \phi \left(I\right),\phi \left(\frac{I}{1+\eta}\right), \cdots, \phi \left(\frac{I}{1+\eta}\right), \phi \left(\frac{\alpha I}{1+\eta}\right),\phi \left( B\right)\right)\\
=& p_{n-1}\left( \phi \left(I\right),\phi \left(\frac{I}{1+\eta}\right), \cdots, \phi \left(\frac{I}{1+\eta}\right), \phi \left(\frac{\alpha I}{1+\eta}\right)\right)\Diamond_{\eta}\phi \left(B\right)\\
=&\phi(\alpha I)\Diamond_{\eta}\phi(B)\\
=&\phi(\alpha I)\Diamond_{\eta}I\\
=&\phi(\alpha I)+\eta\phi(\alpha I)^{\ast}.
\end{aligned}
\eqno{(9)}
$$
This implies that $\phi(\alpha I)+\eta\phi(\alpha I)^{\ast}$ is selfadjoint. So 
$$
\phi(\alpha I)+\eta\phi(\alpha I)^{\ast}=(\phi(\alpha I)+\eta\phi(\alpha I)^{\ast})^{\ast}=\phi(\alpha I)^{\ast}+\overline{\eta}\phi(\alpha I).
$$
Therefore
$$
\phi(\alpha I)^{\ast}=\frac{1-\overline{\eta}}{1-\eta}\phi(\alpha I)=\frac{\overline{\alpha}}{\alpha}\phi(\alpha I).
$$
Putting this in Eq.(9), we get that $\phi(\alpha I)=\alpha I$.

\noindent{\bf Claim 3.2.} For each $A\in\mathcal{A},\phi(A)^{\ast}=\frac{\overline{\alpha}}{\alpha}\phi(A)$ if and only if $A^{\ast}=\frac{\overline{\alpha}}{\alpha}A$.

Let $A$ be in $\mathcal{A}$ such that $A^{\ast}=\frac{\overline{\alpha}}{\alpha}A$. Noting that $1+\eta\frac{\overline{\alpha}}{\alpha}=\frac{1}{\alpha}$, we have that
$$
\begin{aligned}
\phi(A)=&\phi \left( p_n\left( I,\frac{I}{1+\eta}, \cdots, \frac{I}{1+\eta}, \frac{A}{1+\eta},\alpha I\right)\right)\\
=& p_n\left( \phi \left(I\right),\phi \left(\frac{I}{1+\eta}\right), \cdots, \phi \left(\frac{I}{1+\eta}\right), \phi \left(\frac{A}{1+\eta}\right),\phi \left( \alpha I\right)\right)\\
=& p_{n-1}\left( \phi \left(I\right),\phi \left(\frac{I}{1+\eta}\right), \cdots, \phi \left(\frac{I}{1+\eta}\right), \phi \left(\frac{A}{1+\eta}\right)\right)\Diamond_{\eta}\phi \left(\alpha I\right)\\
=&\phi(A)\Diamond_{\eta}\phi(\alpha I)\\
=&\alpha(\phi(A)+\eta\phi(A)^{\ast}.
\end{aligned}
$$
So 
$$
\phi(A)^{\ast}=\frac{1-\alpha}{\eta\alpha}\phi(A)=\frac{\eta\overline{\alpha}}{\eta\alpha}\phi(A)=\frac{\overline{\alpha}}{\alpha}\phi(A).
$$
This proves the sufficiency. The necessity can be obtain by considering $\phi^{-1}$.

\noindent{\bf Claim 3.3.} Let $P$ be in $\mathcal{A}$, then $\frac{1}{\alpha}\phi(\alpha P)$ is a projection in $\mathcal{B}$ if and only if $P$ is a projection in $\mathcal{A}$.

To prove the sufficiency. We suppose that $P$ is a projection in $\mathcal{A}$. Since $(\alpha P)^{\ast}=\frac{\overline{\alpha}}{\alpha}(\alpha P)$, it follows Claim 3.2 that $\phi(\alpha P)^{\ast}=\frac{\overline{\alpha}}{\alpha}\phi((\alpha P))$. Hence $\frac{1}{\alpha}\phi(\alpha P)$ is selfadjoint. Furthermore,
$$
\begin{aligned}
\phi(\alpha P)=&\phi \left( p_n\left( I,\frac{I}{1+\eta}, \cdots, \frac{I}{1+\eta}, \frac{\alpha P}{1+\eta},\alpha P\right)\right)\\
=& p_n\left( \phi \left(I\right),\phi \left(\frac{I}{1+\eta}\right), \cdots, \phi \left(\frac{I}{1+\eta}\right), \phi \left(\frac{\alpha P}{1+\eta}\right),\phi \left( \alpha P\right)\right)\\
=& p_{n-1}\left( \phi \left(I\right),\phi \left(\frac{I}{1+\eta}\right), \cdots, \phi \left(\frac{I}{1+\eta}\right), \phi \left(\frac{\alpha P}{1+\eta}\right)\right)\Diamond_{\eta}\phi \left(\alpha P\right)\\
=&\phi(\alpha P)\Diamond_{\eta}\phi(\alpha P)\\
=&\phi(\alpha P)\phi(\alpha P)+\eta\phi(\alpha P)\phi(\alpha P)^{\ast}\\
=&\phi(\alpha P)^2+\eta\frac{\overline{\alpha}}{\alpha}\phi(\alpha P)^2=\frac{1}{\alpha}\phi(\alpha P)^2\\
\end{aligned}
$$
and then $\left(\frac{1}{\alpha}\phi(\alpha P)\right)^2=\frac{1}{\alpha}\phi(\alpha P)$. So $\frac{1}{\alpha}\phi(\alpha P)$ is a projection. This proves the sufficiency. The necessity can be showed by considering $\phi^{-1}$.
\end{proof}
\vspace{2mm}

\begin{lemma} \label{8}
$\phi$ is multiplicative and hence $\phi(I)=I$.
\end{lemma}

\begin{proof}
Suppose that $\phi$ is multiplicative. Taking $A$ form $\mathcal{A}$ such that $\phi(A)=I$, we have that $\phi(I)=\phi(I)\phi(A)=\phi(A)=I$.

Now we show that $\phi$ is multiplicative. To do this, we fix a projection $Q_{1}$ in $\mathcal{B}$ with $\underline{Q_{1}}=0$ 
and $\overline{Q_{1}}=I$.Then by (1) of Lemma \ref{7}, $P_{1}=\frac{1}{\alpha}\phi^{-1}(\alpha{Q_{1}})$
) is projection in $\mathcal{A}$. It is easy to see that $Q_{1}=\frac{1}{\alpha}\phi(\alpha{P_{1}})$.
Let $P_2=I-P_{1}$ and $Q_2=I-Q_{1}$. Then we obtain from Lemma \ref{7} that $Q_2=\frac{1}{\alpha}\phi(\alpha P_2)$.
Let $\mathcal{A}=\sum_{i,j=1}^2\mathcal{A}_{ij}$ and
$\mathcal{B}=\sum_{i,j=1}^2\mathcal{B}_{ij}$,
where $\mathcal{A}_{ij}=P_{i}\mathcal{A}P_{j}$ and
$\mathcal{B}_{ij}=Q_{i}\mathcal{B}Q_{j}$.

\noindent{\bf Claim 1}. $\phi(\mathcal{A}_{ij})=\mathcal{B}_{ij}$
for $1 \leqslant i\ne j \leqslant 2$.

Let $A_{12}$ be an arbitrary element in $\mathcal{A}_{12}$. Since
$$
\begin{aligned}
\phi(\alpha{A_{12}})& =\phi \left( p_n\left( I,\frac{I}{1+\eta}, \cdots, \frac{I}{1+\eta}, \frac{\alpha P_1}{1+\eta},A_{12}\right)\right)\\
&= p_n\left( \phi \left(I\right),\phi \left(\frac{I}{1+\eta}\right), \cdots, \phi \left(\frac{I}{1+\eta}\right), \phi \left(\frac{\alpha P_1}{1+\eta}\right),\phi \left( A_{12}\right)\right)\\
&= p_{n-1}\left( \phi \left(I\right),\phi \left(\frac{I}{1+\eta}\right), \cdots, \phi \left(\frac{I}{1+\eta}\right), \phi \left(\frac{\alpha P_1}{1+\eta}\right)\right)\Diamond_{\eta}\phi \left(A_{12}\right)\\
&=\phi(\alpha P_1)\Diamond_{\eta}\phi(A_{12})\\
&=\phi(\alpha P_1)\phi(A_{12})+\eta\phi(A_{12})\phi(\alpha P_1)^{\ast}\\
&=\phi(\alpha P_1)\phi(A_{12})+\frac{\overline{\alpha}\eta}{\alpha}\phi(A_{12})\phi(\alpha P_1)\\
&=\alpha{Q_{1}}\phi(A_{12})+\eta\overline{\alpha}\phi(A_{12})Q_{1},
\end{aligned}
$$
we have $Q_{2}\phi(\alpha{A_{12}})Q_{2}=0$. Similarly, we obtain from
$$
\begin{aligned}
\phi(\eta\overline{\alpha}A_{12})&=\phi \left( p_n\left( I,\frac{I}{1+\eta}, \cdots, \frac{I}{1+\eta}, \frac{\alpha P_2}{1+\eta},A_{12}\right)\right)\\
&= p_n\left( \phi \left(I\right),\phi \left(\frac{I}{1+\eta}\right), \cdots, \phi \left(\frac{I}{1+\eta}\right), \phi \left(\frac{\alpha P_2}{1+\eta}\right),\phi \left( A_{12}\right)\right)\\
&= p_{n-1}\left( \phi \left(I\right),\phi \left(\frac{I}{1+\eta}\right), \cdots, \phi \left(\frac{I}{1+\eta}\right), \phi \left(\frac{\alpha P_2}{1+\eta}\right)\right)\Diamond_{\eta}\phi \left(A_{12}\right)\\
&=\phi(\alpha P_2)\Diamond_{\eta}\phi(A_{12})\\
&=\phi(\alpha P_2)\phi(A_{12})+\eta\phi(A_{12})\phi(\alpha P_2)^{\ast}\\
&=\phi(\alpha P_2)\phi(A_{12})+\frac{\overline{\alpha}\eta}{\alpha}\phi(A_{12})\phi(\alpha P_2)\\
&=\alpha{Q_{2}}\phi(A_{12})+\eta\overline{\alpha}\phi(A_{12})Q_{2},
\end{aligned}
$$
we get that $Q_{1}\phi(\eta\overline{\alpha}A_{12})Q_{1}=0$.
Since $A_{12}$ is arbitrary, we have $\phi(A_{12})=B_{12}+B_{21}$ for some $B_{12}\in\mathcal{B}_{12}$ and $B_{21}\in\mathcal{B}_{21}$.

To prove $\phi(A_{12})\in\mathcal{B}_{12}$,we have to show that $B_{21}=0$. Since
$$
\begin{aligned}
0 &=\phi \left( p_n\left( I,\frac{I}{1+\eta}, \cdots, \frac{I}{1+\eta}, \frac{A_{12}}{1+\eta},\alpha P_1\right)\right)\\
&= p_n\left( \phi \left(I\right),\phi \left(\frac{I}{1+\eta}\right), \cdots, \phi \left(\frac{I}{1+\eta}\right), \phi \left(\frac{A_{12}}{1+\eta}\right),\phi \left( \alpha P_1\right)\right)\\
&= p_{n-1}\left( \phi \left(I\right),\phi \left(\frac{I}{1+\eta}\right), \cdots, \phi \left(\frac{I}{1+\eta}\right), \phi \left(\frac{A_{12}}{1+\eta}\right)\right)\Diamond_{\eta}\phi \left( \alpha P_1\right)\\
&=\phi(A_{12})\Diamond_{\eta}\phi(\alpha P_2)\\
&=\phi(A_{12})(\alpha Q_1)+\eta(\alpha Q_1)\phi(A_{12})^{\ast}\\
&=\alpha({B_{12}}+\eta{B_{21}}^*).
\end{aligned}
$$
So we have $B_{21}=0$, which implies $\phi(\mathcal{A}_{12}) \subseteq \mathcal{B}_{12}$. By considering $\phi^{-1}$, we can get $\phi(\mathcal{A}_{12})=\mathcal{B}_{12}$.

Similarly, we have $\phi(\mathcal{A}_{21})=\mathcal{B}_{21}$.

\noindent{\bf Claim 2}.$\phi(\mathcal{A}_{ii}) \subseteq \mathcal{B}_{ii} (i=1,2)$.

Let $A_{ii}$ be an arbitrary element in $\mathcal{A}_{ii}$. Then for $j \neq i$, we have
$$
\begin{aligned}
0 &=\phi \left( p_n\left( I,\frac{I}{1+\eta}, \cdots, \frac{I}{1+\eta}, \frac{\alpha P_j}{1+\eta},A_{ii}\right)\right)\\
&= p_n\left( \phi \left(I\right),\phi \left(\frac{I}{1+\eta}\right), \cdots, \phi \left(\frac{I}{1+\eta}\right), \phi \left(\frac{\alpha P_j}{1+\eta}\right),\phi \left( A_{ii}\right)\right)\\
&= p_{n-1}\left( \phi \left(I\right),\phi \left(\frac{I}{1+\eta}\right), \cdots, \phi \left(\frac{I}{1+\eta}\right), \phi \left(\frac{\alpha P_j}{1+\eta}\right)\right)\Diamond_{\eta}\phi \left( A_{ii}\right)\\
&=\phi(\alpha P_j)\Diamond_{\eta}\phi(A_{ii})\\
&=\phi(\alpha P_j)\phi(A_{ii}+\eta\phi(A_{ii}\phi(\alpha P_j)^{\ast}\\
&=\phi(\alpha P_j)\phi(A_{ii}+\frac{\overline{\alpha}\eta}{\alpha}\phi(A_{ii}\phi(\alpha P_j)\\
&=\alpha{Q_{j}}\phi(A_{ii})+\eta\overline{\alpha}\phi(A_{ii})Q_{j}.
\end{aligned}
$$
which implies that $Q_{j}\phi(A_{ii})Q_{i}=Q_{i}\phi(A_{ii})Q_{j}=0$ and
$\phi(A_{ii})=B_{11}+B_{22}$ for some $B_{11} \in \mathcal{B}_{11}$ and $B_{22} \in \mathcal{B}_{22}$.

For $j \neq i$ and $C_{ij} \in B_{ij}$, we obtain from Claim 1 that $\phi^{-1}(C_{ij}) \in A_{ij}$, thus
$$
\begin{aligned}
0 &=\phi \left( p_n\left( I,\frac{I}{1+\eta}, \cdots, \frac{I}{1+\eta}, \frac{\phi^{-1}(C_{ij})}{1+\eta},A_{ii}\right)\right)\\
&= p_n\left( \phi \left(I\right),\phi \left(\frac{I}{1+\eta}\right), \cdots, \phi \left(\frac{I}{1+\eta}\right), \phi \left(\frac{\phi^{-1}(C_{ij})}{1+\eta}\right),\phi \left( A_{ii}\right)\right)\\
&= p_{n-1}\left( \phi \left(I\right),\phi \left(\frac{I}{1+\eta}\right), \cdots, \phi \left(\frac{I}{1+\eta}\right), \phi \left(\frac{\phi^{-1}(C_{ij})}{1+\eta}\right)\right)\Diamond_{\eta}\phi \left( A_{ii}\right)\\
&=C_{ij}\Diamond_{\eta}\phi(A_{ii})\\
&=C_{ij}\phi(A_{ii}+\eta\phi(A_{ii}C_{ij}^{\ast}\\
&=C_{ij}B_{jj}+\eta B_{jj}){C_{ij}}^{\ast}.
\end{aligned}
$$
It follows from Lemma \ref{2} (1) that $B_{jj}=0$. So we have $\phi(A_{ii})=B_{ii} \subseteq \mathcal{B}_{ii}$.

\noindent{\bf Claim 3}.$\phi$ is multiplicative.

 Since $\phi$ is additive and $\phi(I) = I$. For $A,B \in \mathcal{A}$,we write them as 
 $A = \sum_{i.j=1}^{2}A_{ij}$ and $B =\sum_{i.j=1}^{2}B_{ij}$, where $A_{ij},B_{ij} \in \mathcal{A}_{ij}$. 
 Since $\phi$ is additive, to prove $\phi(AB) = \phi(A)\phi(B)$, it suffices to show that 
 $\phi(A_{ij}B_{kl}) = \phi(A_{ij})\phi(B_{kl})$ for any $i,j,k,l \in \{1,2\}$. If $j \ne k$, then we 
 obtain from Claims 1 and 2 in Section \ref{xxsec4} that $\phi(A_{ij}B_{kl}) =  \phi(A_{ij})\phi(B_{kl}) =0$, thus we 
 just need to consider the cases with $j = k$.

 By the above two claims, we have $\phi(B_{12})\phi(A_{11})^* = 0$, which implies that

 $$
 \begin{aligned}
 \phi(A_{11}B_{12})=&\phi \left( p_n\left( I,\frac{I}{1+\eta}, \cdots, \frac{I}{1+\eta}, \frac{A_{11}}{1+\eta},B_{12}\right)\right)\\
=& p_n\left( \phi \left(I\right),\phi \left(\frac{I}{1+\eta}\right), \cdots, \phi \left(\frac{I}{1+\eta}\right), \phi \left(\frac{A_{11}}{1+\eta}\right),\phi \left(B_{12}\right)\right)\\
=& p_{n-1}\left( \phi \left(I\right),\phi \left(\frac{I}{1+\eta}\right), \cdots, \phi \left(\frac{I}{1+\eta}\right), \phi \left(\frac{A_{11}}{1+\eta}\right)\right)\Diamond_{\eta}\phi \left( B_{12}\right)\\
=&\phi(A_{11})\Diamond_{\eta}\phi(B_{12})\\
=&\phi(A_{11})\phi(B_{12})+\eta\phi(B_{12})\phi(A_{11})^{*}\\
=&\phi(A_{11})\phi(B_{12}).
 \end{aligned}
 $$

Similarly, we can prove that $\phi(A_{22}B_{21})=\phi(A_{22})\phi(B_{21})$.

 For $D_{12} \in \mathcal{B}_{12}$, we have $C_{12}=\phi^{-1}(D_{12}) \in \mathcal{A}_{12}$ by Claim 1. Therefore
 $$
 \phi(A_{11}B_{11})D_{12}=\phi(A_{11}B_{11}C_{12})=
 \phi(A_{11})\phi(B_{11}C_{12})=\phi(A_{11})\phi(B_{11})D_{12}.
 $$
for all $D_{12} \in \mathcal{B}_{12}$. Since $\underline{Q_{1}} = 0$ and $\overline{Q_{1}} =I$,we obtain from Lemma \ref{2} and Claim 2 in Section \ref{xxsec4} that $\phi(A_{11}B_{11}) = \phi(A_{11})\phi(B_{11})$. 

Similarly, we have $\phi(A_{22}B_{22}) = \phi(A_{22})\phi(B_{22})$.

 Since  $\phi(B_{21}){\phi(A_{12})}^{*} = 0$ by Claim 1 in Section \ref{xxsec4}, we have
$$
 \begin{aligned}
  \phi(A_{12}B_{21})=&\phi \left( p_n\left( I,\frac{I}{1+\eta}, \cdots, \frac{I}{1+\eta}, \frac{A_{12}}{1+\eta},B_{21}\right)\right)\\
=& p_n\left( \phi \left(I\right),\phi \left(\frac{I}{1+\eta}\right), \cdots, \phi \left(\frac{I}{1+\eta}\right), \phi \left(\frac{A_{12}}{1+\eta}\right),\phi \left(B_{21}\right)\right)\\
=& p_{n-1}\left( \phi \left(I\right),\phi \left(\frac{I}{1+\eta}\right), \cdots, \phi \left(\frac{I}{1+\eta}\right), \phi \left(\frac{A_{12}}{1+\eta}\right)\right)\Diamond_{\eta}\phi \left( B_{21}\right)\\
=&\phi(A_{12})\Diamond_{\eta}\phi(B_{21})\\
=&\phi(A_{12})\phi(B_{21})+\eta\phi(B_{21})\phi(A_{12})^{*}\\
=&\phi(A_{12})\phi(B_{21}).
 \end{aligned}
 $$
 
 Similarly we have $\phi(A_{21}B_{12}) = \phi(A_{21}\phi(B_{12})$.

 For $D_{21} \in \mathcal{B}_{21}$, we have $C_{21}=\phi^{-1}(D_{21}) \in \mathcal{A}_{21}$ by Claim 1 in Section \ref{xxsec4}. Therefore
 $$
 \phi(A_{12}B_{22})D_{21}=\phi(A_{12}B_{22}C_{21})=
 \phi(A_{12})\phi(B_{22}C_{21})=\phi(A_{12})\phi(B_{22})D_{21}.
 $$
 for all $D_{21} \in \mathcal{B}_{21}$. Since $\underline{Q_{1}} = 0$ and $\overline{Q_{1}} = I$,we know by Lemma 
 \ref{2} and Claim 2 in Section \ref{xxsec4} that $\phi(A_{12}B_{22}) = \phi(A_{12})\phi(B_{22})$.

 Similarly, we have $\phi(A_{22}B_{21}) = \phi(A_{22})\phi(B_{21})$. 

\end{proof}

\begin{lemma}\label{xxsec3.4} 
We have
\begin{enumerate}
\item[(1)] $\phi(\alpha{A}) = \alpha\phi(A)$ for each $A \in \mathcal{A}$;

\item[(2)] If $A \in \mathcal{A}$ is selfadjoint, then $\phi(A)$ is selfadjoint;

\item[(3)]  $\phi$ is real linear.
\end{enumerate}
\end{lemma}

 \begin{proof}
 (1) For $A \in \mathcal{A}$, we know by the above two lemmas that
 $$
 \phi(\alpha{A}) = \phi((\alpha{I})A) = \phi(\alpha{I})\phi(A) = \alpha\phi(A).
 $$

 (2) Suppose that $A \in \mathcal{A}$ is selfadjoint, then
 $(\alpha{A})^{\ast} = \frac{\overline{\alpha}}{\alpha}(\alpha{A})$.
 Thus we know from (1) and Lemma \ref{7} (2) that
 $$
 \overline{\alpha}\phi(A)^\ast = (\alpha(\phi(A))^{\ast}=
 \phi(\alpha{A})^\ast=\frac{\overline{\alpha}}{\alpha}
 \phi(\alpha{A}) = \overline{\alpha}\phi(A).
 $$
 so $\phi(A)$ is selfadjoint.

 (3) Let $A$ be a positive element in $\mathcal{A}$. Then we have $A = C^{2}$ for some selfadjoint element $C \in \mathcal{A}$. Hence
 $\phi(A) = {\phi(C)}^{2}$. Since $\phi(C)$ is selfadjoint, $\phi(A)$ is positive, which implies that $\phi$ preserves positive elements.

 Now let $a$ be a real number. Choose sequences $\{b_{n}\}$ and $\{c_{n}\}$ of rational numbers 
 such that $b_{n}  \leqslant a  \leqslant c_{n}$ for all $n$ and $\text{lim}_{n \to \infty}b_{n}= \text{lim}_{n \to \infty}c_{n} = a$.
 Since $b_{n}I  \leqslant aI  \leqslant c_{n}I$ and $\phi$ preserves positive elements, we
 know that $b_{n}I  \leqslant \phi(aI)  \leqslant c_{n}I$. Since $\mathcal{A}$ is a von 
 Neumann algebra, after taking the limit, we know that $\phi(aI) = aI$. Hence 
 for $A \in \mathcal{A}$, we have $\phi(aA) = \phi((aI)A) = \phi(aI)\phi(A) = a\phi(A)$.
 \end{proof}

\begin{lemma}\label{xxsec3.5} 
Suppose that  $\eta \notin \mathbb{R}$. Then $\phi$ is linear.
 \end{lemma}

 \begin{proof} 
 From the proof of Lemma \ref{7}, we know that $\alpha \notin\mathbb{R}$. 
 Let $\alpha = a+b\text{i}$ for some $a,b \in \mathbb{R}$. Then $b \neq 0$. For $A \in \mathcal{A}$, 
 we obtain from Lemma \ref{xxsec3.4} (3) that
 $$
 a\phi(A)+b\phi(\text{i}A) = \phi((a+b\text{i})A) = (a+b\text{i})\phi(A).
 $$
 Thus we have $\phi(\text{i}A) =\text{i}\phi(A)$. This together with Lemma \ref{xxsec3.4} shows that $\phi$ is linear.  
 \end{proof}

\begin{lemma}\label{xxsec3.6}
 For all $A \in \mathcal{A},\phi({A}^{\ast}) = {\phi(A)}^{\ast}$.
 \end{lemma}

 \begin{proof} 
 For $A \in \mathcal{A}$, we know by the above two lemmas and the additivity of $\phi$ that
 $$
 \begin{aligned}
\phi(A)+\eta\phi(A^{\ast}) &= \phi(A)+\phi(\eta{A}^{\ast})\\
 &= \phi(AI+\eta I{A}^{\ast})\\
&=\phi \left( p_n\left(I,\frac{I}{1+\eta}, \cdots, \frac{I}{1+\eta}, \frac{A}{1+\eta},I\right)\right)\\
&= p_n\left( \phi \left(I\right),\phi \left(\frac{I}{1+\eta}\right), \cdots, \phi \left(\frac{I}{1+\eta}\right), \phi \left(\frac{A}{1+\eta}\right),\phi \left(I\right)\right)\\
&= p_{n-1}\left( \phi \left(I\right),\phi \left(\frac{I}{1+\eta}\right), \cdots, \phi \left(\frac{I}{1+\eta}\right), \phi \left(\frac{A}{1+\eta}\right)\right)\Diamond_{\eta}\phi \left( I\right)\\
 & = \phi(A)\Diamond_{\eta}I\\
 &=\phi(A)+\eta{\phi(A)}^{\ast}.
 \end{aligned}
 $$
 Thus we have $\phi(A^{\ast}) = {\phi(A)}^{\ast}$.  
 \end{proof}

\begin{lemma}\label{xxsec3.7} 
$\eta \in \mathbb{R}$, and there is a central projection $E \in \mathcal{A}$ such that the restriction of $\phi$ to $\mathcal{A}E$ is linear and the restriction of $\phi$ to $\mathcal{A}(I-E)$ is conjugate linear.
\end{lemma}

\begin{proof}
 By Lemma \ref{8}, ${\phi(\text{i}I)}^{2} =\phi({\text{i} I}^{2}) = -\phi(I) = -I$. By Lemma \ref{xxsec3.6}, 
 ${\phi(\text{i}I)}^{\ast} =\phi({\text{i} I}^{\ast}) = -\phi(\text{i} I)$.  Let
 $F = \frac{I-\text{i}\phi(\text{i}I)}{2}$. Then it is easy to verify that $F$ is a central projection in $\mathcal{B}$. Let $E = {\phi(F)}^{-1}$. 
 From Lemma \ref{xxsec3.4} (2) show that $E^{\ast} = E$. 

On the other hand, by Lemma \ref{8}
  $$
 E^{2} = ({\phi^{-1})(F)}^{2} = {\phi^{-1}(F)}{\phi^{-1}(F)} = {\phi^{-1}(F^{2})} = {\phi^{-1}(F)} =E.
 $$
So $E$ is a projection. For any $B \in \mathcal{B}$, $BF = FB$, then ${\phi^{-1} (BF)}= {\phi^{-1}(FB)}$.
This together with Lemma \ref{8} shows that $E \in \mathcal{Z}(\mathcal{A})$. So $E$ is a 
central projection in $\mathcal{A}$. Moreover, for $A \in \mathcal{A}$, we have
 $$
 \phi(\text{i}AE) = \phi(A)\phi(E)\phi(\text{i}I) = \text{i}\phi(A)F = \text{i}\phi(AE),
 $$
 and
 $$
 \phi(\text{i}A(I-E)) = \phi(A)\phi(I-E)\phi(\text{i}I) =\text{i}\phi(A)(I-F) = -\text{i}\phi(A(I-E)).
 $$
 Thus the restriction of $\phi$ to $\mathcal{A}E$ is linear and the restriction of $\phi$ to $\mathcal{A}(I-E)$ is conjugate linear. 
\end{proof}
 The proof of Theorem \ref{Th2} follows now from the above results.
\end{proof}

\bigskip

\section{Potential Topics for the Future Research}
\label{xxsec5}

\end{document}